\documentclass[
]{amsart}
\usepackage{amssymb, amsthm, amsmath, pifont} 
\usepackage[usenames]{color}
\usepackage{amsfonts}
\usepackage{enumerate}
\usepackage{pdfpages}
\usepackage[all]{xypic}
\usepackage{graphicx}
\usepackage{psfrag}
\usepackage{verbatim}
\usepackage{pstricks}
\usepackage[latin1]{inputenc}
\usepackage{mathrsfs}
\usepackage[all,knot]{xy}
\newtheorem{theorem}{Theorem}[section]
\newtheorem{corollary}[theorem]{Corollary}

\newtheorem{lemma}[theorem]{Lemma}
\newtheorem{question}[theorem]{Question}
\newtheorem{example}[theorem]{Example}
\newtheorem{definition}[theorem]{Definition}
\newtheorem*{intdef}{Definition}
\newtheorem{remark}[theorem]{Remark}
\newtheorem{prop}[theorem]{Proposition}
\newtheorem{claim}[theorem]{Claim}
\newtheorem{obs}[theorem]{Observation}

\newenvironment{Example}{\begin{example}\rm}{\end{example}}

\newenvironment{Remark}{\begin{remark}\rm}{\end{remark}}
\newenvironment{Claim}{\begin{claim}\rm}{\end{claim}}

\newenvironment{Question}{\begin{question}\rm}{\end{question}}
\def\et{\;\mbox{and}\;}

\def\HU{
\widetilde{\Upsilon}}
\def\HI{
\widetilde{I}}

\def\wr{{\rm{wr}}}
\def\l{{{\ell}}}
\def\fdtc{{\omega}}

\def\for{\quad\mbox{for }}

\def\et{\quad\mbox{and}\quad}

\def\epsilon{\varepsilon}

\def\R{\mathbb{R}}
\def\Z{\mathbb{Z}}

\begin{document}
\title{Braids with as many full twists as strands realize the braid index
}
\author{Peter Feller}

\address{ETH Z\"urich, R\"amistrasse 101, 8092 Z\"urich, Switzerland}
\email{peter.feller@math.ch}
\author{Diana Hubbard}

\address{Brooklyn College, CUNY, 2900 Bedford Avenue, Brooklyn, NY  11210 USA}
\email{diana.hubbard@brooklyn.cuny.edu}

\begin{abstract}
  We characterize the fractional Dehn twist coefficient of a braid in terms of a slope of the homogenization of the Upsilon function, where Upsilon is the function-valued concordance homomorphism defined by Ozsv\'ath, Stipsicz, and Szab\'o. We use this characterization to prove that $n$-braids with fractional Dehn twist coefficient larger than $n-1$ realize the braid index of their closure. As a consequence, we are able to prove a conjecture of Malyutin and Netsvetaev stating that $n$-times twisted braids realize the braid index of their closure. We provide examples that address the optimality of our results. The paper ends with an appendix about the homogenization of knot concordance homomorphisms.
\end{abstract}
\maketitle
\section{Introduction}

A \emph{braid} or \emph{$n$-braid} is an element of \emph{Artin's braid group on $n$-strands} $B_n$~\cite{Artin_TheorieDerZoepfe}, which can be presented as
\[B_n=\left\langle a_1,\cdots,a_{n-1}\;\middle|\;a_ia_j=a_ja_i\text{ for }|i-j|\geq 2,a_ia_{i+1}a_i=a_{i+1}a_ia_{i+1}\right\rangle.\]

Our main result about braids connects two notions from different perspectives on braid theory. On one hand, viewing braids as mapping classes of the $n$ punctured closed disk $D_{n}$ leads to the notion of \emph{the fractional Dehn twist coefficient} $\fdtc(\beta)$ of the conjugacy class of a braid $\beta$: a rational number which roughly speaking measures how much the mapping class twists along the boundary of $D_{n}$ as one performs an isotopy to its canonical representative.  On the other hand, \emph{links}---oriented and closed smooth $1$-submanifolds of $S^3$ considered up to ambient isotopy---can be studied as the closures of braids; see Figure \ref{fig:braid_closure}.
Indeed, by Alexander's theorem~\cite{Alexander_23_ALemmaOnSystemsOfKnottedCurves} all links arise as closures of braids, making the following well-defined: 
the \emph{braid index} of a link $L$ is the smallest positive integer $n$ such that there exists an $n$-braid with closure $L$.

 It is a long-standing open problem to find an algorithm that determines the braid index of a given link;
 compare to Birman and Brendle's survey~\cite[Open Problem 1]{Birman_Brendle_braidssurvey}.
 With the exception of certain families (for instance, see~\cite{Murasugi_BraidIndexAlternatingLinks} and~\cite{Franks_Williams_87_BraidsAndTheJonesPolynomial}), little is known about the braid indices of knots and links. One of the most famous results about the braid index is the Morton-Franks-Williams (MFW) inequality, which gives bounds on the braid index in terms of the Jones/HOMFLY-PT polynomial (\cite{Franks_Williams_87_BraidsAndTheJonesPolynomial}, \cite{Morton_SeifertCircles}, \cite{Morton_PolynomialsFromBraids}). In \cite{Birman_Brendle_braidssurvey}, Birman and Brendle observed that this was, to their knowledge, the only ``general result" about the braid index. While the MFW inequality is sharp on all but five of the prime knots with up to ten crossings (\cite{Jones_MFW}), it is not sharp for infinitely many knots and links, and furthermore, Kawamuro showed that the defect between the MFW bound and the braid index can be arbitrarily large (see \cite{Kawamuro_braidindex}, \cite{Kawamuro_KR_MFW}, \cite{Elrifai_thesis}).

We relate the braid index and the fractional Dehn twist coefficient as follows.
\begin{theorem}\label{thmintro:FDTCofNonminimalBraidIsBounded}
For any integer $n \geq 2$, every $n$-braid $\beta$ with $|\fdtc(\beta)|>n-1$ realizes the braid index of its closure. In other words, for every $n$-braid $\beta$ such that there exists an $(n-1)$-braid with isotopic link closure, we have $|\fdtc(\beta)| \leq n-1$.
\end{theorem}


\begin{figure}[h]
\centering
\includegraphics[scale=0.3]{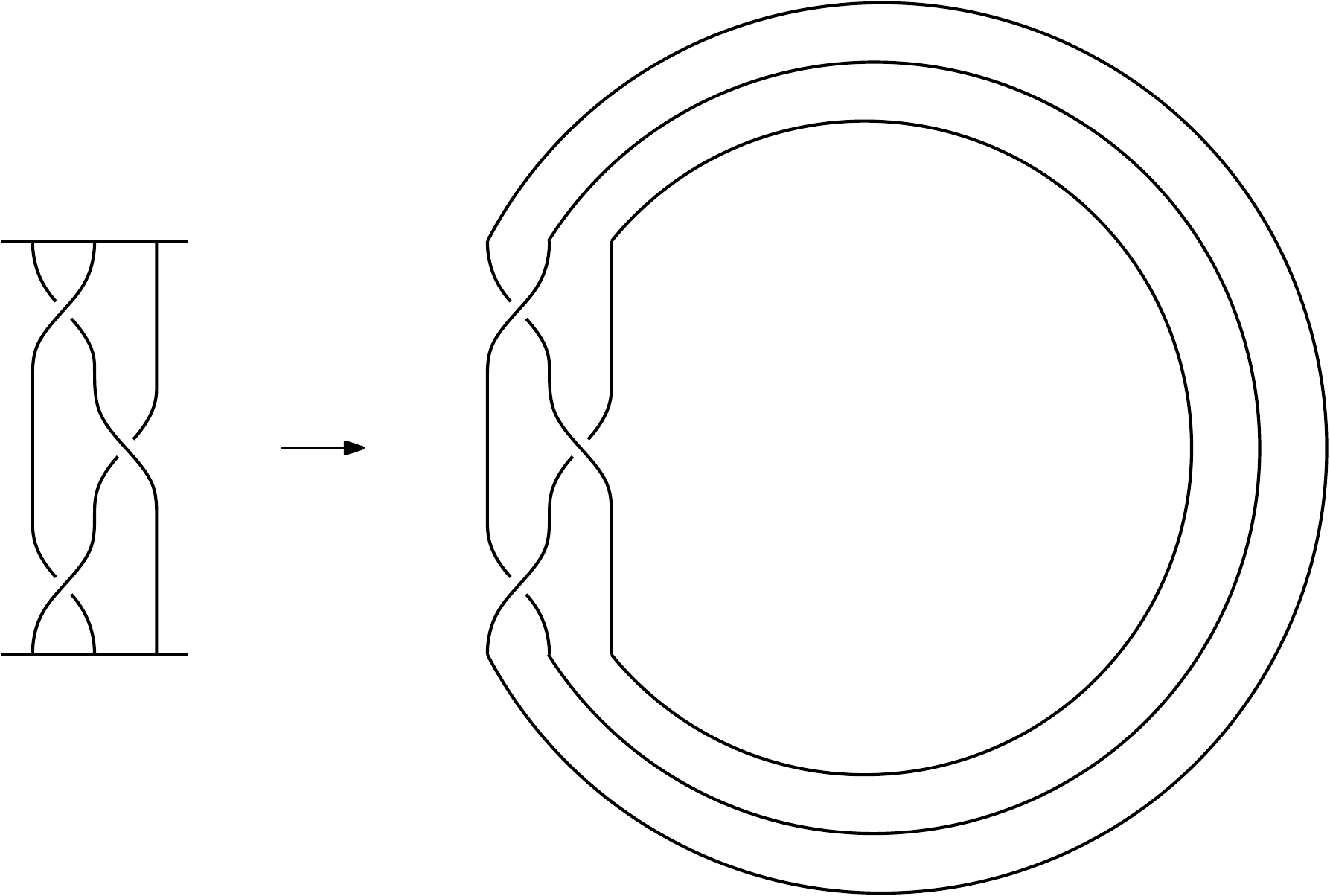}
\caption{On the left, the geometric braid corresponding to the $3$-braid $a_{1}a_{2}^{-1}a_{1}$. {In general, $a_{i}$ in the $n$-stranded braid group $B_{n}$ corresponds to a positive half-twist between the $i$'th and $i+1$'st strands.}
On the right, its closure $\widehat{a_{1}a_{2}^{-1}a_{1}}$. Braids are oriented upwards and their closures are oriented accordingly.}
\label{fig:braid_closure}
\end{figure}

We show in Section \ref{sec:Examples} (see Example~\ref{Ex:elrifai_examples}) that this result determines the braid index for infinitely many examples where the MFW inequality fails to be sharp. Furthermore, in Section~\ref{sec:Examples} we discuss examples, originally discovered by Malyutin and Netsvetaev in~\cite{MalyutinNetsvetaev_03}, of $n$-braids with fractional Dehn twist coefficient $n-2$ that do not realize the braid indices of their closures.  These examples show that Theorem~\ref{thmintro:FDTCofNonminimalBraidIsBounded} is very close to optimal, with the possibility that the bound could be improved to $n-2$. The main tool for the proof of Theorem~\ref{thmintro:FDTCofNonminimalBraidIsBounded} is Theorem~\ref{thmintro:FDTCviaUpsilon} given below---a characterization of $\fdtc(\beta)$ in terms of Ozsv\'ath, Stipsicz, and Szab\'o's $\Upsilon$-invariant for knots, which is defined using the Heegaard Floer knot complex $CFK^\infty(K)$. {(\emph{Knots} are links consisting of a single connected component.)} Surprisingly, our proof of Theorem~\ref{thmintro:FDTCofNonminimalBraidIsBounded}, which is a purely 3-dimensional result, uses the concordance properties of the $\Upsilon$-invariant (in other words, its 4-dimensional aspects; see Section~\ref{sec:PropofhomoofUpsilon} and Appendix~\ref{app:homogenizationofconcordancehomos}). Before we discuss Theorem~\ref{thmintro:FDTCviaUpsilon}, we briefly recall a description of the fractional Dehn twist coefficient from Malyutin~\cite{Malyutin_Twistnumber} via a Thurston-type order on the braid group due to Dehornoy~\cite{Dehornoy_WhyAreBraidsOrderable} and we use Theorem~\ref{thmintro:FDTCofNonminimalBraidIsBounded} to resolve a conjecture by Malyutin and Netsvetaev.

A braid $\beta$ is said to be \emph{Dehornoy positive}, denoted by $\beta\succ 1$, if it can be written as a braid word that, for some integer $1\leq i<n$, contains a braid generator $a_i$ but no $a_i^{-1}$ or any generators $a_j^{\pm 1}$ for $j<i$. Dehornoy showed that this gives a well-defined, left invariant, total order $\succ$ on $B_n$ by setting $\alpha\prec\beta$ to mean $\alpha^{-1}\beta\succ1$. The \emph{Dehornoy floor} $\lfloor\beta\rfloor$ is the unique integer $m$ such that $(\Delta^2)^{m+1}\succ\beta\succeq(\Delta^2)^{m}$ where $\Delta^{2} = (a_{1}\cdots a_{n-1})^{n}$ is the full twist on $n$ strands. The fractional Dehn twist coefficient equals the homogenization of the Dehornoy floor, i.e.~for any $\beta$, ~$\fdtc(\beta)=\lim_{k\to\infty}\frac{\lfloor\beta^k\rfloor}{k}$; see~\cite{Malyutin_Twistnumber}. Using this description of the fractional Dehn twist coefficient, Theorem~\ref{thmintro:FDTCofNonminimalBraidIsBounded} allows us to conclude the following: 

\begin{corollary}[Compare to Conjecture 7.4 in~\cite{MalyutinNetsvetaev_03}]\label{corintro:ConjMalyutinNetsetaev}
Fix an integer $n\geq2$. If an $n$-braid $\beta$ satisfies $\Delta^{2n}\preceq\beta$ or $\beta\preceq\Delta^{-2n}$, then the closure of $\beta$ does not arise as the closure of a braid on $n-1$ or fewer strands.
\end{corollary}

In~\cite{MalyutinNetsvetaev_03}, Malyutin and Netsvetaev used work of Birman and Menasco in~\cite{BirmanMenasco_StabilizationI} (specifically, their Markov theorem without stabilization) to show that for every $n\geq 2$ there exist a constant $r_n$ such that, if an $n$-braid $\beta$ satisfies $\Delta^{2r_{n}}\preceq\beta$ or $\beta\preceq\Delta^{-2r_{n}}$, then the closure of $\beta$ does not arise as the closure of a braid on $n-1$ or fewer strands. Their proof is based on a counting argument which does not yield the constant $r_n$ explicitly. However, they showed that $r_{n} \geq n-1$ (\cite{MalyutinNetsvetaev_03}, Example 7.5), crucially observing that $r_{n}$ must increase with the number of strands,
and they conjectured that $r_{n} = n$ works (see~\cite[Conjecture 7.4]{MalyutinNetsvetaev_03}. Our approach allows to prove that conjecture; see Corollary~\ref{corintro:ConjMalyutinNetsetaev}. 
We describe the characterization of $\fdtc(\beta)$ in terms of the homogenization of $\Upsilon$ next.

In~\cite{OSS_2014}, Ozsv\'ath, Stipsicz, and Szab\'o associate to a knot $K$ 
a piecewise linear function $\Upsilon_K\colon[0,1]\to\R$.
Its homogenization is an invariant of 
braids defined by
\[\HU_\beta(t)=\lim_{k\to\infty}\frac{\Upsilon_{\widehat{\beta^k\varepsilon_{\beta^k}}}}{k},\]
where $\varepsilon_{\beta^k}$ is a shortest possible (as a word in the Artin generators and their inverses) $n$-braid such that the closure of $\beta^k\varepsilon_{\beta^k}$ is a knot rather than a link.
The homogenization of many (concordance) knot invariants (including Ozsv\'ath and Szab\'o's $\tau$ invariant that is generalized by $\Upsilon$) are completely determined by the writhe, where the \emph{writhe} $\wr(\beta)$ of a braid $\beta$ equals the exponent sum of its braid word; see~\cite{Brandenbursky_11}. One instance of this is $\HU$ of $n$-braids for $t\leq \frac{2}{n}$: for an $n$-braid $\beta$, we have $\HU_{\beta}(t)=-t\frac{\wr(\beta)}{2}$ for $t\leq\frac{2}{n}$; see~\cite
{FellerKrcatovich_16_OnCobBraidIndexAndUpsilon}. Our main result on $\HU$ is that  $\widetilde{\Upsilon}_\beta(t)$ is also linear on $[\frac{2}{n},\frac{2}{n-1}]$ and the change of slope at $\frac{2}{n}$ equals $\fdtc(\beta)n$:
\begin{theorem}\label{thmintro:FDTCviaUpsilon}
  Fix an integer $n\geq 2$. For all $n$-braids $\beta$, we have
  \[\widetilde{\Upsilon}_\beta(t)=-t\frac{\wr(\beta)}{2}+\fdtc(\beta) n(t-\frac{2}{n}) \for \frac{2}{n}\leq t\leq \min\{\frac{2}{n-1},1\}
.\]
\end{theorem}

It is known that $\omega(\beta)$ can only take on values in a certain set of rational numbers (see Proposition~\ref{fdtcproperties}) which allows us to conclude in Corollary~\ref{cor:upsilonrational} that the slope change at $\frac{2}{n}$ can also only take on certain rational values.  A priori, it does not seem clear 
that one should even expect $\HU$ to be linear on $[\frac{2}{n},\frac{2}{n-1}]$. In contrast to this, Gambaudo and Ghys studied the homogenization of the $\omega$-signature function, which, properly normalized, is also determined by the writhe on $[0,\frac{2}{n}]$: it agrees with $\HU$ on $[0,\frac{2}{n}]$; see~\cite{GambaudoGhys_BraidsSignatures} and~\cite
{FellerKrcatovich_16_OnCobBraidIndexAndUpsilon}. However, in general the homogenization of the $\omega$-signature function behaves non-linearly on $[\frac{2}{n},\frac{2}{n-1}]$; for example, it is not linear on $[\frac{2}{3},1]$ for the $3$-braid $a_1^3a_2^7$ (see ~Example~4.6 in~\cite{FellerKrcatovich_16_OnCobBraidIndexAndUpsilon}).

Finally, we end with a short discussion of a specific corollary of the MFW inequality in order to compare and contrast it with Theorem \ref{thmintro:FDTCofNonminimalBraidIsBounded}. This corollary is the fact that any positive braid that can be written as the product of a single positive full twist with another positive braid realizes the braid index~\cite
{Franks_Williams_87_BraidsAndTheJonesPolynomial}. (Recall that a braid is \emph{positive} if it has a word representative that only uses positive powers of the Artin generators.) The reader may ask why, in contrast, our results require the number of twists to increase with the number of strands. The reason is that the positivity condition in the MFW corollary is more restrictive than our conditions. For example, our results can apply to non-positive braids. For instance, an $n$-braid of the form $\Delta^{2n}(a_{2}a_{3}\cdots a_{n-1})^{-k}$ has fractional Dehn twist coefficient $n$ and is neither positive nor even quasipositive for large enough $k$. Furthermore, we point to the existence of a family of positive braids with an increasing number of twists (more precisely, with increasing fractional Dehn twist coefficient) that do not realize the braid index; see Example~\ref{optimalityexample}.

\subsection*{Organization}
We describe the structure of the paper and the proofs of our main results.
In Section~\ref{sec:FDTC}, we provide background on the fractional Dehn twist coefficient. In Section~\ref{sec:PropofhomoofUpsilon}, we recall properties of $\Upsilon$, and provide properties of its homogenization.
In Section~\ref{sec:FDTCviaUpsilon}, we prove Theorem~\ref{thmintro:FDTCviaUpsilon}. We use the description of the fractional Dehn twist coefficient in terms of the Dehornoy floor, a lower bound on $\HU$ for Dehornoy positive braids in terms of the writhe (Lemma~\ref{lemma:beta>1}), 
the characterization of $\Upsilon$ for torus knots $T_{n,n+1}$, and linearity of $\HU$ on $[0,\frac{2}{n}]$ and $[0,\frac{2}{n-1}]$ for $n$-braids and $(n-1)$-braids, respectively. 
In Section~\ref{sec:UpsilonandBraidIndex}, we derive Theorem~\ref{thmintro:FDTCofNonminimalBraidIsBounded} 
from Theorem~\ref{thmintro:FDTCviaUpsilon} by employing the fact that the difference between $\HU$ of an $n$-braid and an $m$-braid with the same closure is bounded by $t\frac{n+m-2}{2}$ (Proposition~\ref{prop:Upsilon_alpha-Upsilon_beta}) and the generalized Jones conjecture as proven by Dynnikov and Prasolov~\cite{DynnikovPrasolov_13} (compare also with~\cite{LaFountainMenasco_14}).
In Section~\ref{sec:Examples}, we provide examples that show that Theorem~\ref{thmintro:FDTCofNonminimalBraidIsBounded} is essentially optimal.
In Section~\ref{sec:questions}, we collect some questions.
Finally, in Appendix~\ref{app:homogenizationofconcordancehomos} we prove properties of homogenizations of concordance homomorphisms that specialize to properties of the homogenization of $\Upsilon$ provided in Section~\ref{sec:PropofhomoofUpsilon}.

\subsection*{Acknowledgments}
The first author thanks Benjamin Hennion and Kristian Moi for helpful discussions and the second author thanks David Krcatovich for a helpful conversation about Upsilon. The first author thanks the Max Planck Institute for Mathematics in Bonn for their support and hospitality. The second author was supported in part by NSF RTG grant 1045119.
\section{Background on the fractional Dehn twist coefficient}\label{sec:FDTC}

As we will see in more detail towards the end of this section, the fractional Dehn twist coefficient can be defined in many different ways, and in fact can be defined not only for braids but also for mapping classes of general surfaces with boundary. It first appeared in the literature in the work of Gabai and Oertel on essential laminations of $3$-manifolds (see~\cite{Gabai_EssentialLaminations3Manifolds}), though there it is referred to in very different language.  The definition that is most useful to us comes from Dehornoy's order on the braid group, and is due to Malyutin in~\cite{Malyutin_Twistnumber}. The advantage of this point of view is that Dehornoy's order provides a concrete characterization of the positivity of a braid in terms of its word in the Artin generators.

Recall from the introduction that a braid $\beta$ is said to be \emph{Dehornoy positive}, denoted by $\beta\succ 1$, if it can be written as a braid word that contains a braid generator $a_i$ for some integer $1\leq i<n$ but no $a_i^{-1}$ and no $a_j^{\pm 1}$ for $j<i$. We then say that $\alpha\prec\beta$ if $\alpha^{-1}\beta\succ1$. In \cite{Dehornoy_94_Braidgroups} (see also \cite{Dehornoy_WhyAreBraidsOrderable}), Dehornoy proved that this is a well-defined left-invariant total order $\succ$ on $B_n$.

While Dehornoy was the first to establish the existence of a left-invariant total order on the braid group, many more orders are now known coming from geometric considerations.  In~\cite{Fenn_Orderingbraidgroups}, the five authors give a method for constructing orders on $B_{n}$ that involves comparing the action of braids on diagrams of curves drawn on the punctured disk $D_{n}$. In~\cite{Short_OrderingsmappingclassgroupsafterThurston}, Short and Wiest describe and classify more orderings on $B_{n}$ (originally due to Thurston). These orderings come from equipping $D_{n}$ with a hyperbolic structure, and considering the action of braids on the boundary of its universal cover, viewed in $\mathbb{H}^{2}$, together with its limit points on the circle at infinity. Both of these perspectives can give rise to orders not only on $B_{n}$ but more generally on mapping class groups of surfaces with boundary (see \cite{Rourke_OrderAutomaticMappingClassGroups} and \cite{Short_OrderingsmappingclassgroupsafterThurston}).

While it is possible to prove that Dehornoy's ordering is in fact total and left-invariant using entirely combinatorial and algebraic tools, the order has natural geometric content. Indeed, it can be recovered as an order coming from both the curve diagram perspective and the Thurston perspective, and many of the properties of Dehornoy's order are more or less immediate from the geometric point of view. While the geometric perspective is in some sense more natural, working with Dehornoy's order directly makes many of our computations more straightforward.

Recall that $\Delta^{2} \in B_{n}$ is the element $(a_{1} \cdots a_{n-1})^{n}$; it corresponds to a full twist around the boundary of $D_{n}$ and commutes with every other element in $B_{n}$. Dehornoy's order on $B_{n}$ now allows us to define the following: the \emph{Dehornoy floor} $\lfloor\beta\rfloor$ is the unique integer $m$ such that $(\Delta^2)^{m+1}\succ\beta\succeq(\Delta^2)^{m}$. The intuition here is that the Dehornoy floor gives a measurement of how many positive full twists can be extracted from a braid so that the remainder is still non-negative in the order. We now can define the fractional Dehn twist coefficient of a braid $\beta \in B_{n}$, denoted $\omega(\beta)$, as follows (\cite{Malyutin_Twistnumber}):
$$\fdtc(\beta)=\lim_{k\to\infty}\frac{\lfloor\beta^k\rfloor}{k}$$

One can prove that this is well-defined in a self-contained way using the fact that Dehornoy's floor is a quasimorphism.

We collect in the following proposition some properties of the fractional Dehn twist coefficient that are relevant for this paper:

\begin{prop}\label{fdtcproperties} [\cite{Malyutin_Twistnumber}, \cite{ItoKawamuro_OpenBookFoliations}]  For any $\alpha, \beta \in B_{n}$, we have:
\begin{enumerate}[a)]
\item (Quasimorphism) $|\omega(\alpha\beta) - \omega(\alpha) - \omega(\beta) | \leq 1$
\item (Homogeneity) $\omega(\alpha^n) = n\omega(\alpha)$
\item (Behavior under full twists) $\omega(\Delta^{2}\alpha) = \omega(\alpha) + 1$
\item (Conjugacy invariant) $\omega(\alpha) = \omega(\beta\alpha\beta^{-1})$
\item $\omega(\alpha)$ is rational, and in fact $\{\omega(\alpha) | \alpha \in B_{n}\} = \{\frac{p}{q} | \, p \in \mathbb{Z}, q \in \mathbb{Z}, 1 \leq q \leq n\}$.

\end{enumerate}

\end{prop}

Note that Dehornoy's floor is not a conjugacy invariant, but the fractional Dehn twist coefficient is. Properties $(a)-(d)$ of the fractional Dehn twist coefficient can be proved directly from the definition in terms of the Dehornoy floor. In fact, Property $(d)$ is a straightforward consequence of properties $(a)-(b)$. Malyutin's proof of Property $(e)$ requires a different, but equivalent, definition of the fractional Dehn twist coefficient and involves considering cases depending on the Nielsen-Thurston classification of the braid in question.

Very briefly, to define the fractional Dehn twist coefficient in this alternate way, one can consider the compactification of the universal cover of $D_{n}$ embedded in $\mathbb{H}^{2}$, use the action of the lift of $\beta$ to this universal cover to define a map $\Theta: B_{n} \to \widetilde{Homeo^{+}(S^{1})}$, and define $\omega(\beta)$ to be the translation number of $\Theta(\beta)$. For a more thorough discussion, see~\cite{Malyutin_Twistnumber}, \cite{ItoKawamuro_OpenBookFoliations}, and~\cite{Plamenevskaya_RightVeering}.  For yet another alternate and equivalent definition that demonstrates more clearly that the fractional Dehn twist coefficient is measuring the amount of (signed) twisting a braid realizes around $\partial D_{n}$, see~\cite{HondaKazezMatic_RightVeeringI}, \cite{KazezRoberts_FDTC}, and \cite{ItoKawamuro_OpenBookFoliations}.  Both of these alternate definitions generalize easily beyond braids to elements in mapping class groups of surfaces with boundary.

The fractional Dehn twist coefficient has been extensively studied in the context of contact topology and open book decompositions (see for instance \cite{HondaKazezMatic_RightVeeringI}, \cite{HondaKazezMatic_RightVeeringII}, \cite{KazezRoberts_FDTC}, \cite{ItoKawamuro_OpenBookFoliations}, \cite{BaldwinEtnyre_admissiblesurgery}, \cite{HeddenMark_FloerFDTC}) and of course in the context of classical braids (\cite{Malyutin_Twistnumber}, \cite{MalyutinNetsvetaev_03}). Relationships have also been explored between the fractional Dehn twist coefficient and monoids in the mapping class group (\cite{Etnyre_MonoidsMappingClassGroup}, \cite{ItoKawamuro_OnAQuestion}), classical knot theory (\cite{KazezRoberts_FDTC}), and homological invariants of knots and 3-manifolds (\cite{HeddenMark_FloerFDTC}, \cite{Plamenevskaya_RightVeering}) .

\section{The homogenization of Upsilon}\label{sec:PropofhomoofUpsilon}

In this section, we discuss properties of the homogenization of Ozsv\'ath, Stipsicz, and Szab\'o's $\Upsilon$.
Rather than recalling the definition of $\Upsilon$ using the $CFK^\infty(K)$ knot Floer complex, we will only recall some of its properties and work with those. This is appropriate since our results would hold for any other invariant that satisfies these properties.
While no other such invariants are known as of this writing\footnote{Recently, Grigsby-Wehrli-Licata \cite{GrigsbyLicataWehrli_16} and Lewark-Lobb \cite{LewarkLobb_17_Upsilonlike} defined $\Upsilon$-type invariants using annular Khovanov cohomology and higher $sl_N$-Khovanov-Rozansky cohomologies, respectively. However, neither of these invariants fit the framework of an $\Upsilon$-type invariant as needed here.}, one might hope for such invariants to be found in the future, in a similar way as Ozsv\'ath and Szab\'o's $\tau$ invariant (which is generalized by $\Upsilon$) led to the discovery of invariants with similar properties (e.g.~the Rasmussen $s$-invariant).
In addition to the original article~\cite{OSS_2014},  Livingston's note~\cite{Livingston_NotesOnUpsilon} is a good and short reference for the definition and properties of $\Upsilon$.

We delay most of the proofs of the statements in this section to the end of the paper (see Appendix~\ref{app:homogenizationofconcordancehomos}) for the following reasons: first, these proofs are somewhat long but standard arguments using language from knot concordance theory and do not constitute the core of the argument of our main results.
Additionally, these proofs are best given in a general setting of homogenization of concordance knot invariants rather than the specific case of $\Upsilon$, so, for future reference, an independent appendix seems more appropriate.

\subsection*{Background on the concordance homomorphism $\Upsilon$}
Recall that two links $K$ and $L$ are called \emph{concordant} if there exists an oriented smooth embedding of a disjoint union of annuli in $S^3\times[0,1]$ such that the oriented boundary is $K\times\{0\}\cup L^{\textrm{rev}}\times\{1\}$, where $L^{\textrm{rev}}$ denotes the result of reversing the orientation of $L$. 
Knots up to concordance form a group, called the \emph{concordance group}:
\[\mathcal{C}=(\{\text{concordance classes of knots}\},\#),\]
where $\#$ denotes the operation induced by connected sum of knots. {For all knots $K$, the knot $-K$ given by taking the mirror of $K$ and reversing orientation, represents the inverse of the class of $K$ in $\mathcal{C}$.}

Ozsv\'ath, Stipsicz, and Szab\'o associate to a knot $K$ (in fact, to its concordance class) a piecewise linear function $\Upsilon_K\colon[0,2]\to\R$, which turns out to be a strong tool in detecting free subgroups and free summands of the concordance group; see~\cite{OSS_2014}.
In what follows we consider $\Upsilon_K$ as a function on $[0,1]$ by restriction without losing any information since $\Upsilon(t)=\Upsilon(2-t)$ for all $t\in[0,2]$; see~Proposition~1.2 in \cite{OSS_2014}.

We summarize all the properties of $\Upsilon$ needed in this paper in the following proposition.
For this, recall that the \emph{$3$-genus} or \emph{genus} $g(K)$ of a knot $K$ is the smallest genus among smooth oriented surfaces in $S^3$ with boundary $K$. Similarly, the \emph{smooth $4$-ball genus} or \emph{slice genus} $g_4(K)$ of a knot $K$ is the smallest genus among smoothly embedded surfaces in the $4$-ball $B^4$ with boundary $K\subset S^3=\partial B^4$.
For positive coprime integers $p$ and $q$, the knot given as the closure of the $p$-braid $(a_1a_2\cdots a_{p-1})^q$ is denoted $T_{p,q}$ and called the $(p,q)$-\emph{torus knot}.

\begin{prop}[\cite{OSS_2014}]\label{prop:PropOfU} Let $\rm{PL}[0,1]$ denote the group (with respect to addition in the target) of piecewise-linear, $\R$-valued, continuous functions on $[0,1]$. There exists a group homomorphism, the \emph{Upsilon-invariant},
\[\Upsilon\colon\mathcal{C}\to \rm{PL}[0,1]\] that satisfies the following properties:
\begin{itemize}
\item
  \cite[Theorem~1.11]{OSS_2014}: For all knots $K$ and all $t\in[0,1]$, $\left|\Upsilon_K(t)\right|\leq tg_4(K)$.
  \item
  \cite[Theorem~1.13]{OSS_2014}: For all knots $K$, the absolute value of the slopes of $\Upsilon_K$ is bounded above by $g(K)$.
  \item
  \cite[Theorem~1.15 and Proposition~6.3]{OSS_2014}: For positive integers $n$ and $k$, 
\[\Upsilon_{T_{n,nk+1}}=-tg_4(T_{n,nk+1})=-tg(T_{n,nk+1})=-t\frac{n(n-1)k}{2}\for t\leq \frac{2}{n}\et\]
\[\Upsilon_{T_{n,nk+1}}=-tg_4(T_{n,nk+1})+nk(t-\frac{2}{n})\for \frac{2}{n}\leq t\leq \min\{\frac{2}{n-1},1\}.\qed\]

\end{itemize}

\end{prop}

\subsection*{The homogenization of $\Upsilon$}
The Upsilon-invariant can be used to construct an invariant of (conjugacy classes of) braids, called the \emph{homogenization of $\Upsilon$}, as follows:
\[\HU\colon B_n\to \rm{Cont}[0,1], \beta\mapsto \lim_{k\to\infty}\frac{\Upsilon_{\widehat{\beta^k\varepsilon_{\beta^k}}}}{k},\]
where $\varepsilon_{\beta^k}$ is a shortest possible (as a word in the generators $a_i$) $n$-braid such that the closure of $\beta^k\varepsilon_{\beta^k}$ is a knot rather than a link; concretely, $\varepsilon_{\beta^k}$ can be chosen to be the product of at most $n-1$ generators $a_i$. Here $\rm{Cont}[0,1]$ denotes the real-valued continuous functions on $[0,1]$.

In~\cite{Brandenbursky_11}, Brandenbursky studied this construction in a more general context: for any knot invariant $I$ that descends to a homomorphism $I\colon \mathcal{C}\to \R$ with $|I(K)|\leq t_I g_4(K)$ for all knots $K$ and some real constant $t_I$, he showed there is a well-defined (independent of the choice of shortest possible $\varepsilon_{\beta^k}$)
 map 
\[\HI\colon B_n\to \R, \beta\mapsto \lim_{k\to\infty}\frac{I\left(\widehat{\beta^k\varepsilon_{\beta^k}}\right)}{k}.\]
For a fixed $t\in[0,1]$, $\Upsilon$ fits into this setting with $I=\Upsilon(t)$ and $t_{I}=t$ by Proposition~\ref{prop:PropOfU}.

We summarize properties of $\HU$ that hold for every fixed $t\in[0,1]$ and a fixed number of strands $n\geq1$; see Lemma~\ref{lemma:PropofHI} for a proof.
\begin{lemma}\label{lem:PropofHU}
\begin{enumerate}[I)]
  \item\label{item:HUofgenerator} For all $n$-braids $\beta$ and all $1\leq i\leq n-1$, $\left|\HU_{\beta a_i}(t)-\HU_{\beta}(t)\right|\leq \frac{t}{2}$.
  \item\label{item:HUhasDefectt(n-1)} For all $n$-braids $\alpha$ and $\beta$, $|\HU_{\alpha\beta}(t)-\HU_{\alpha}(t)-\HU_{\beta}(t)|\leq t(n-1)$. If $\alpha$ and $\beta$ commute, e.g.~if $\alpha$ is the $n$-stranded full twist $\Delta^2$ or a power of $\beta$, then
      $\HU_{\alpha\beta}(t)=\HU_{\alpha}(t)+\HU_{\beta}(t)$.
  \item \label{item:HUofUnions}If an $n$-braid $\beta$ is given as the disjoint union (see Figure \ref{fig:disjoint_union}) of braids $\beta_1$, \ldots, $\beta_l$ on $n_1$, \ldots , $n_l$ strands, respectively, then $\HU_\beta(t)=\sum_{i=1}^l \HU_{\beta_i}(t)$.\qed
\end{enumerate}
\end{lemma}

\begin{figure}[h]
\centering
\includegraphics[scale=0.7]{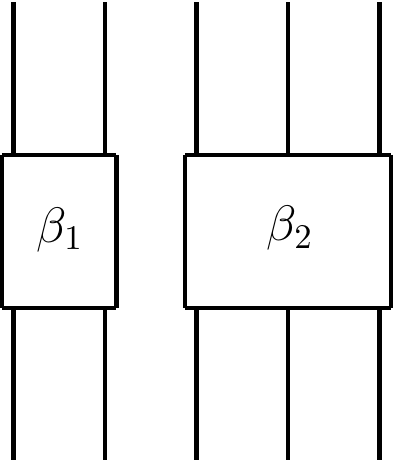}
\caption{The disjoint union of the $2$-braid $\beta_{1}$ and the $3$-braid $\beta_{2}$.}
\label{fig:disjoint_union}
\end{figure}

Our proof of Theorem~\ref{thmintro:FDTCofNonminimalBraidIsBounded} will use the characterization of the fractional Dehn twist coefficient in terms of $\HU$ provided in Theorem~\ref{thmintro:FDTCviaUpsilon} (and proved in Section~\ref{sec:FDTCviaUpsilon}), the generalized Jones conjecture as proven by~\cite[Theorem~9]{DynnikovPrasolov_13}, and the following bound on the difference of $\HU$ of braids that have isotopic or concordant closures.
The following is Proposition~\ref{prop:HI(alpha)-HI(beta)} for $I=\Upsilon(t)$.
\begin{prop}\label{prop:Upsilon_alpha-Upsilon_beta}
Fix positive integers $n$ and $m$. If an $n$-braid $\beta$ and an $m$-braid $\alpha$ have isotopic links as their closure (or, more generally, concordant links as their closure), then
\[\left|\HU_\beta(t)-\HU_\alpha(t)\right|\leq t\frac{n-1+m-1}{2}\for t\in [0,1].\qed\]
\end{prop}

The value of $\Upsilon$ for torus knots $T_{n,kn+1}$ and $t\leq \frac{2}{n}$ (see Proposition~\ref{prop:PropOfU}) implies the following:
for an $n$-braid, we have $\HU(t)=-t\frac{\wr(\beta)}{2}$ for $t\leq\frac{2}{n}$; see~\cite[Corollary~4.2]{FellerKrcatovich_16_OnCobBraidIndexAndUpsilon} or Lemma~\ref{lemma:HIwhenSBIissharp}. Combined with Lemma~\ref{lem:PropofHU}.\ref{item:HUofUnions}, we can state this as follows:
\begin{lemma}\label{lemma:HUislinearforsmallt}
For all $n$-braids $\beta$, let $m\leq n$ be the smallest integer such that $\beta$ is the disjoint union of braids on $m$ or fewer strands. Then,
\[\HU_\beta(t)=-t\frac{wr(\beta)}{2}\for t\leq \min\left\{1,\frac{2}{m}\right\}.\qed\]
\end{lemma}

Next, we discuss properties of $\HU$ as a function depending on $t$. 
As a consequence of the fact that the slopes of $\Upsilon$ are bounded by the $3$-genus (see~\cite[Theorem~1.13]{OSS_2014}), one finds that, for a fixed braid, $\HU$ is Lipschitz continuous:
\begin{prop}\label{prop:HUisLipschitz}
For all $n$-braids $\beta$, we have
\[|\HU_\beta(t)-\HU_\beta(s)|\leq |t-s|\frac{\l(\beta)}{2},\]
where $\l(\beta)$ denotes the minimal number of generators $a_i$ and their inverses needed to write $\beta$.
\end{prop}
While Proposition~\ref{prop:HUisLipschitz} is not used in the rest of the paper, it brings us to ask about the regularity of $\HU$; see Question~\ref{q:HUisPL}.

\begin{proof}[Proof of Proposition~\ref{prop:HUisLipschitz}]
We note that
\begin{equation}\label{eq:g<=l/2}g\left(\widehat{\beta^k\varepsilon_{\beta^k}}\right)\leq \frac{\l(\beta^k\varepsilon_{\beta^k})-(n-1)}{2}\leq \frac{k\l(\beta)+n-1-(n-1)}{2}\end{equation}
 since applying the Seifert algorithm to a standard diagram of the closure of an $n$-braid given by a braid word of length $l$ yields a Seifert surface of genus $\frac{l-n+1}{2}$.
 Thus,
\begin{align*}\left|\HU_\beta(t)-\HU_\beta(s)\right|
&=\lim_{k\to\infty}\left|\frac{\Upsilon_{\beta^k\varepsilon_{\beta^k}}(t)-\Upsilon_{\beta^k\varepsilon_{\beta^k}}(s)}{k}\right|
\\&\leq\liminf_{k\to\infty}|t-s|\frac{g\left(\widehat{\beta^k\varepsilon_{\beta^k}}\right)}{k}
\\&\leq |t-s|\frac{\l(\beta)}{2},\end{align*}
where in the second and third line we use that $\Upsilon_K$ is $g(K)$-Lipschitz continous for all knots (Proposition~\ref{prop:PropOfU}~\cite[Theorem~1.13]{OSS_2014}) and~\eqref{eq:g<=l/2}, respectively.
\end{proof}

\section[FDTC as a slope of the homogenization of Upsilon]{The fractional Dehn twist coefficient as a slope of the homogenization of Upsilon}\label{sec:FDTCviaUpsilon}
In this section, we study the homogenization of $\Upsilon$ for a fixed integer $n\geq2$.
We describe $\widetilde{\Upsilon}_\beta(t)$ for $t\in[\frac{2}{n},\frac{2}{n-1}]$ and all $n$-braids $\beta$. It turns out that
$\widetilde{\Upsilon}_\beta$ is linear on $[\frac{2}{n},\frac{2}{n-1}]$ with slope $\frac{-\wr(\beta)}{2}+n\fdtc(\beta)$, which is the content of Theorem~\ref{thmintro:FDTCviaUpsilon}. 

Let $\preceq$ ($\prec$) denote Dehornoy's (strict) total order.
\begin{lemma}\label{lemma:beta>1}
Let $\beta$ be an $n$-braid. If $\beta\succeq 1$, then $\widetilde{\Upsilon}_\beta(t)\geq-t\frac{\wr(\beta)}{2}$ for $t\leq\frac{2}{n-1}$.
\end{lemma}
\begin{proof}
Without loss of generality, we may and do assume that $\beta$ cannot be written as a braid word without $a_1$ or $a_1^{-1}$, since otherwise $\widetilde{\Upsilon}_\beta(t)=-t\frac{\wr(\beta)}{2}$ for $t\leq\frac{2}{n-1}$ by Lemma~\ref{lemma:HUislinearforsmallt}.

By the definition of $\beta\succ 1$, we have
\[\beta=\alpha_0a_1\alpha_1a_1\cdots\alpha_{l-1}a_1\alpha_{l}\]
for an integer $l\geq 1$, where the $\alpha_i$ are braids `only involving strands 2 to $n$'; i.e.~the $\alpha_i$ can be given by braid words that do not contain $a_1$ or $a_1^{-1}$.

Let $\beta'$ be the $n$-braid given by $\alpha_0\alpha_1\cdots\alpha_{l-1}\alpha_{l}$.
Since $\beta'$ can be obtained from $\beta$ by deleting $l$ generators, $l$ times applying Lemma~\ref{lem:PropofHU}.\ref{item:HUofgenerator} yields
$\left|\widetilde{\Upsilon}_{\beta'}(t)-\widetilde{\Upsilon}_{\beta}(t)\right|\leq t\frac{l}{2}$. Therefore,
\[\widetilde{\Upsilon}_{\beta}(t)\geq \widetilde{\Upsilon}_{\beta'}(t)-t\frac{l}{2}=-t\frac{\wr(\beta')}{2}-t\frac{l}{2}=-t\frac{\wr(\beta)}{2}\]
for $t\leq \frac{2}{n-1}$, where the first equality uses Lemma~\ref{lemma:HUislinearforsmallt}. 
\end{proof}

Let $\Delta^2$ denote the $n$-braid $(a_1a_2\cdots a_{n-1})^n$, called the \emph{positive full twist}.
We have \begin{equation}\label{eq:HomUpsofFullTwist}
\HU_{\Delta^2\beta}=\HU_\beta+\HU_{\Delta^2}
=\HU_\beta+\lim_{k\to\infty}\frac{\Upsilon_{T_{n,kn+1}}}{k}
=\HU_\beta+\Upsilon_{T_{n,n+1}};\end{equation} where the first equality holds since $\Delta^2$ is in the center of $B_n$ (see Lemma~\ref{lem:PropofHU}\ref{item:HUhasDefectt(n-1)}), the second equality follows from choosing $\epsilon_{(\Delta^2)^k}$ to be $a_1a_2\cdots a_{n-1}$ in the definition of $\HU$ and noting that then the closure of $(\Delta^2)^k\epsilon_{(\Delta^2)^k}$ is the torus knot $T(n,kn+1)$, and the third equality is immediate from $\Upsilon_{T_{n,kn+1}}=k\Upsilon_{T_{n,n+1}}$ (see \cite[Proposition~2.2]{FellerKrcatovich_16_OnCobBraidIndexAndUpsilon}). Using~\eqref{eq:HomUpsofFullTwist} we establish the following.
\begin{corollary}\label{cor:beta>Delta}
Let $\beta$ be an $n$-braid. If $\beta\succeq \Delta^{2m}$, then $\widetilde{\Upsilon}_\beta(t)\geq-t\frac{\wr(\beta)}{2}+mn(t-\frac{2}{n})$ for $t\in[\frac{2}{n},\frac{2}{n-1}]$.
\end{corollary}
\begin{proof}
By definition, $\beta\succeq \Delta^{2m}$ means $\Delta^{-2m}\beta\succeq 1$. For $t\in[\frac{2}{n},\frac{2}{n-1}]$, we calculate
\begin{align*}
-t\frac{wr(\Delta^{-2m})}{2}-t\frac{wr(\beta)}{2}&=-t\frac{wr(\Delta^{-2m}\beta)}{2}\\&\leq\HU_{\Delta^{-2m}\beta}(t)
\\&=\HU_{\Delta^{-2m}}(t)+\HU_{\beta}(t)
\\&=-m\Upsilon_{T_{n,n+1}}(t)+\HU_{\beta}(t)
\\&=-m\left(-t\frac{n(n-1)}{2}+n(t-\frac{2}{n})\right)+\HU_{\beta}(t)
\\&=-t\frac{wr(\Delta^{-2m})}{2}-nm(t-\frac{2}{n})+\HU_{\beta}(t),
\end{align*}
where Lemma~\ref{lemma:beta>1} is used in the second line, \eqref{eq:HomUpsofFullTwist} is used in the fouth line, and the value for $\Upsilon_{T_{n,n+1}}$ as provided in~\cite[Proposition~6.3]{OSS_2014} (see Proposition~\ref{prop:PropOfU}) is used in the fifth line. This yields the desired lower bound for $\HU_{\beta}$.
\end{proof}

We note that, since $\beta\preceq 1$ if and only if $\beta^{-1}\succeq 1$ and $\HU_{\beta^{-1}}=-\HU_{\beta}$, Lemma~\ref{lemma:beta>1} and Corollary~\ref{cor:beta>Delta} also hold when replacing $\succeq$ and $\geq$ by $\preceq$ and $\leq$, respectively. This allows to conclude the following:

\begin{prop}\label{prop:HUofbeta}
Let $\beta$ be an $n$-braid. Assume $\beta$ has Dehornoy floor $\lfloor\beta\rfloor=m\in\Z$, i.e.~$\Delta^{2m+2}~\succ\beta\succeq \Delta^{2m}$. Then, for $t\in[\frac{2}{n},\frac{2}{n-1}]$, we have
\[-t\frac{\wr(\beta)}{2}+(m+1)n(t-\frac{2}{n})\geq\HU_\beta(t)\geq-t\frac{\wr(\beta)}{2}+mn(t-\frac{2}{n})
.\qed\]
\end{prop}

Using the characterization of the fractional Dehn twist coefficient as
$\fdtc(\beta)=\lim_{k\to\infty}\frac{\lfloor\beta^k\rfloor}{k}$, Proposition~\ref{prop:HUofbeta} yields
Theorem~\ref{thmintro:FDTCviaUpsilon}, which we restate as follows.

\newtheorem*{thmCharOfFDTC}{Theorem~\ref{thmintro:FDTCviaUpsilon}}
\begin{thmCharOfFDTC}
  Fix an integer $n\geq 2$. For all $n$-braids $\beta$, we have
  \[\widetilde{\Upsilon}_\beta(t)=\left\{\begin{array}{c}
-t\frac{\wr(\beta)}{2}\for t\leq \frac{2}{n}\\
-t\frac{\wr(\beta)}{2}+\fdtc(\beta) n(t-\frac{2}{n}) \for \frac{2}{n}\leq t\leq \frac{2}{n-1}
\end{array}\right..\]
\end{thmCharOfFDTC}
In other words, $\widetilde{\Upsilon}_\beta(t)$ is linear on $[0,\frac{2}{n}]$ and $[\frac{2}{n},\frac{2}{n-1}]$ with change of slope equal to $\fdtc(\beta)n$ at $\frac{2}{n}$.
\begin{proof}[Proof of Theorem~\ref{thmintro:FDTCviaUpsilon}]
By Lemma~\ref{lemma:HUislinearforsmallt}, $\HU_\beta(t)$ is linear on $[0,\frac{2}{n}]$ with slope $-\frac{\wr(\beta)}{2}$. Thus, we only discuss the case $t\in [\frac{2}{n},\frac{2}{n-1}]$.
For any integer $k>0$, we have
\begin{equation}\label{eq:HUofbeta^k}
\frac{-t\frac{\wr(\beta^k)}{2}+(\lfloor\beta^k\rfloor+1)n(t-\frac{2}{n})}{k}\geq\frac{\HU_{\beta^k}(t)}{k}
\geq\frac{-t\frac{\wr(\beta^k)}{2}+\lfloor\beta^k\rfloor n(t-\frac{2}{n})}{k},\end{equation}
by Proposition~\ref{prop:HUofbeta}.

Note that the writhe of a braid is homogeneous and, by construction, $\HU$ is homogeneous, i.e.~$\wr(\beta^k)=k\wr(\beta)$ and $\HU_{\beta^k}=k\HU_{\beta}$, respectively, for all integers $k$ and all braids $\beta$.
With this we rewrite~\eqref{eq:HUofbeta^k} as
\[-t\frac{\wr(\beta)}{2}+\frac{(\lfloor\beta^k\rfloor)n(t-\frac{2}{n})}{k}+\frac{n(t-\frac{2}{n})}{k}\geq\HU_{\beta}(t)\geq-t\frac{\wr(\beta)}{2}+\frac{\lfloor\beta^k\rfloor n(t-\frac{2}{n})}{k},\] from which the result follows by taking the limit $k\to\infty$.
\end{proof}

Theorem \ref{thmintro:FDTCviaUpsilon} combined with property (e) of Proposition \ref{fdtcproperties} immediately yields:

\begin{corollary}\label{cor:upsilonrational} For every braid group $B_{n}$, the set of all possible changes in slope of $\HU(t)$ at $\frac{2}{n}$ is precisely $\{\frac{np}{q} | p \in \mathbb{Z}, q \in \mathbb{Z}, 1 \leq q \leq n\}$. Each of these values is realized by some braid in $B_{n}$.\qed
\end{corollary}
For instance, the $4$-braid $A = a_{1}a_{2}a_{3}a_{3}$ has fractional Dehn twist coefficient $\frac{1}{3}$ (since one can first see using braid relations that $A^{3} = \Delta^{2}$, and then apply (b) and (c) from Proposition \ref{fdtcproperties}), and so $\HU_{A}(t)$ has slope change $\frac{4}{3}$ at $\frac{1}{2}$. The $5$-braid $B = a_{1}a_{2}a_{3}a_{4}a_{1}a_{2}$ also has fractional Dehn twist coefficient $\frac{1}{3}$ and so $\HU_{B}(t)$ has slope change $\frac{5}{3}$ at $\frac{2}{5}$. Here we calculate the fractional Dehn twist coefficient of $B$ by first observing that $\omega(a_{1}a_{2}a_{3}a_{4}) \leq \omega(B) \leq \omega((a_{1}a_{2}a_{3}a_{4})^2)$ (this is an application of Lemma 5.2 in ~\cite{Malyutin_Twistnumber}), which implies that $\frac{1}{5} \leq \omega(B) \leq \frac{2}{5}$ as $(a_{1}a_{2}a_{3}a_{4})^5 = \Delta^{2}$. Then combining the fact that $B$ is a pseudo-Anosov braid~\cite{Ham_Song_PseudoAnosov5Braid} and Malyutin's restrictions  in~\cite{Malyutin_Twistnumber} on which values of the fractional Dehn twist coefficient are realized by pseudo-Anosov braids yields the calculation. Notice that Corollary~\ref{cor:upsilonrational} is in contrast to the situation for $\Upsilon$, which only has integral slopes (and hence only integral changes in slope).

\section{Homogenization of Upsilon and braid index}\label{sec:UpsilonandBraidIndex}
Based on the characterization of $\fdtc$ in terms of the slope of $\HU$ (Theorem~\ref{thmintro:FDTCviaUpsilon}), we derive Theorem~\ref{thmintro:FDTCofNonminimalBraidIsBounded} about the braid index and $\fdtc$. A key element of our proof is the generalized Jones Conjecture as proven by Dynnikov and Prasolov ~\cite[Theorem~9]{DynnikovPrasolov_13}, which we quote here for reference:

\newtheorem*{GeneralizedJonesConjecture}{Theorem \cite[Theorem~9]{DynnikovPrasolov_13}}

\begin{GeneralizedJonesConjecture} Suppose braids $\beta_{1} \in B_{m}$ and $\beta_{2} \in B_{n}$ represent the same class of oriented links and $\beta_{1}$ has the smallest possible number of strands in that class. Then
$$ | \wr(\beta_{2}) - \wr(\beta_{1}) | \leq n-m.$$
\end{GeneralizedJonesConjecture}

We now recall Theorem ~\ref{thmintro:FDTCofNonminimalBraidIsBounded} before proving it.
\newtheorem*{thmFDTCandBraidIndex}{Theorem~\ref{thmintro:FDTCofNonminimalBraidIsBounded}}
\begin{thmFDTCandBraidIndex}
Fix an integer $n\geq 2$. For any $n$-braid $\beta$ such that there exists an $(n-1)$-braid with isotopic closure, we have $|\fdtc(\beta)|\leq n-1$.
\end{thmFDTCandBraidIndex}

\begin{proof}
Let $\alpha$ be an $m$-braid such that $\alpha$ and $\beta$ have the same closure, where $m\leq n-1$ is the braid index of the closure of $\beta$.

By the generalized Jones Conjecture as proven by Dynnikov and Prasolov~\cite[Theorem~9]{DynnikovPrasolov_13}, we have $|\wr(\beta)-\wr(\alpha)|\leq n-m$. 

We use Proposition~\ref{prop:Upsilon_alpha-Upsilon_beta} with $t=\frac{2}{n-1}$ to find
\begin{align*}\left|\left(-\wr(\beta)+2n\fdtc(\beta)(1-\frac{n-1}{n})\right)+\wr(\alpha)\right|
&=\left|2t^{-1}\left(\HU_\beta(t)-\HU_\alpha(t)\right)\right|
\\&\leq n+m-2,\end{align*}
where the equality and the inequality are given by Theorem~\ref{thmintro:FDTCviaUpsilon} and Proposition~\ref{prop:Upsilon_alpha-Upsilon_beta}, respectively. Therefore, we have
\[2|\fdtc(\beta)|\leq |\wr(\alpha)-\wr(\beta)|+n+m-2\leq 2n-2,\]
as claimed.
\end{proof}
\begin{Remark}
In terms of $\HU$, Theorem~\ref{thmintro:FDTCofNonminimalBraidIsBounded} states that, given an $n$-braid for which the absolute value of the slope change of $\HU$ at $\frac{2}{n}$ is strictly larger than $n(n-1)$, said braid realises the braid index of its closure.
\end{Remark}
\begin{corollary}(Compare to Conjecture 7.4 of~\cite{MalyutinNetsvetaev_03})
Fix an integer $n\geq2$. If an $n$-braid $\beta$ satisfies $\Delta^{2n}\preceq\beta$ or $\beta\preceq\Delta^{-2n}$, then the closure of $\beta$ does not arise as the closure of a braid on $n-1$ or fewer strands.
\end{corollary}
\begin{proof}
If $\Delta^{2n}\preceq\beta$, then $\Delta^{2nk}\preceq\beta^k$ for all positive integers $k$, and, thus, $\fdtc(\beta)\geq n>n-1$.
Similarly, if $\beta\preceq\Delta^{-2n}$, then $\beta^k\preceq\Delta^{-2nk}\prec\Delta^{(-2n)k+2}$ for all positive integers $k$, and, thus,
$\fdtc(\beta)\leq -n<-(n-1)$. Consequently, the corollary follows from Theorem~\ref{thmintro:FDTCofNonminimalBraidIsBounded}.
\end{proof}

\section{Examples and Optimality}\label{sec:Examples}

The following example shows that Theorem \ref{thmintro:FDTCofNonminimalBraidIsBounded} is (very close to) optimal.

\begin{Example}\label{optimalityexample}

For positive integers $n,m\geq 2$, let $\beta_{n,m}$ be the $n$-braid $(\delta\delta^\Delta)^{m-1}\delta$, where
\[\delta=a_1a_2\cdots a_{n-1}\et\delta^\Delta=a_{n-1}a_{n-2}\cdots a_1.\] We calculate below that $\fdtc(\beta_{n,m})=m-1$. (Note that this should intuitively be clear, since in $\beta_{n,m}$ the first strand is wrapping $m-1$ times around the rest.)

It was observed in~\cite{MalyutinNetsvetaev_03} that
the closures of $\beta_{n,m}$ and $\beta_{m,n}$ are isotopic (briefly, their observation was that these are the same link with respect to different braid axes: see Figure 2 in ~\cite{MalyutinNetsvetaev_03}); thus, when $n>m$, we have that $\beta_{n,m}$ does not realize the braid index of its closure. In particular, $\beta_{n,n-1}$ is an $n$-braid with fractional Dehn twist coefficient $n-2$ that does not realize the braid index of its closure. On the other hand, Theorem~\ref{thmintro:FDTCofNonminimalBraidIsBounded} implies that, if $n<m$, then $\beta_{n,m}$ does realize the braid index of its closure. This leaves the question whether $\beta_{n,n}$ realizes the braid index of its closure. It turns out that this is the case; see Proposition~\ref{prop:betanm} below.

To show that $\fdtc(\beta_{n,m})=m-1$,  we rewrite $\beta_{n,m}$ as follows: \[\beta_{n,m}=\Delta^{2m-2}\Delta_{2,\cdots,n}^{-2m+2}\delta,\]
where by $\Delta^2_{2,\cdots,n}$ we mean the $n$-braid $(a_{2}\cdots a_{n-1})^{n-1}$ (that is, it is the full twist on the last $n-1$ strands).
Similarly, we denote by $\Delta^2_{1,\cdots,n-1}$ the $n$-braid $(a_{1}\cdots a_{n-2})^{n-1}$ (that is, it is the full twist on the first $n-1$ strands), and note that $\Delta^{2l}_{2,\cdots,n}\delta=\delta\Delta^{2l}_{1,\cdots,n-1}$ for all integers $l$. For all positive integers $k$, we calculate \begin{align*}(\Delta^{(2m-2)k})^{-1}\beta_{n,m}^k&=(\Delta_{2,\cdots,n}^{-2m-2}\delta)^k\succ1\et\\
(\Delta^{(2m-2)k+2})^{-1}\beta_{n,m}^k&=\Delta^{-2}(\Delta_{2,\cdots,n}^{-2m-2}\delta)^k
\\&=\Delta^{-2}(\delta\Delta_{1,\cdots,n-1}^{-2m-2})^{k}
\\&=\Delta^{-2}\delta(\Delta_{1,\cdots,n-1}^{-2m-2}\delta)^{k-1}\Delta_{1,\cdots,n-1}^{-2m-2}
\\&\prec 1,\end{align*}
where in the last line we use that $\Delta^{-2}\delta$ and $\Delta_{1,\cdots,n-1}^{-2m-2}\delta$ can be written as a braid word containing $a_1^{-1}$ but no $a_1$. Consequently, we have \[\Delta^{2(m-1)k+2}\succ\beta_{n,m}^k\succ\Delta^{2(m-1)k}\]
for all positive integers $k$ and, thus, $\fdtc(\beta_{n,m})=m-1$.


\end{Example}
We remind the reader that a braid is called \emph{positive} if it can be given as a word in which only generators $a_i$ (but no $a_i^{-1}$) feature. Similarly, a braid is called \emph{quasipositive} if it can be written as a word in conjugates of generators $a_i$. A knot is called \emph{quasipositive} if it arises as the closure of a quasipositive braid.
\begin{prop}\label{prop:betanm}
If a knot $K$ is the closure of an $n$-braid of the form
 \[\alpha_1\delta\beta_1\delta^\Delta\alpha_2\delta\beta_2\cdots\alpha_{n-1}\delta\beta_{n-1}\delta^\Delta\alpha_{n}\delta\beta_{n},\]
      where the $\alpha_j$ and $\beta_j$ are (possibly trivial) quasipositive $n$-braids,
then $K$ has braid index $n$.

In fact, any quasipositive knot $K'$ (more generally, knot $K'$ that is the closure of a braid on which the slice-Bennequin inequality is sharp) concordant to $K$ has braid index at least $n$.
\end{prop}
Here the slice-Bennequin inequality being sharp on a braid $n$-braid $\beta$ whose closure is a knot means
$\frac{\wr(\beta)-(n-1)}{2}=g_4\left(\widehat{\beta}\right)$. In particular, one has $g_4\left(\widehat{\beta}\right)=\tau\left(\widehat{\beta}\right)$ by~\cite[Corollary~11]{Livingston_Comp}, where $\tau$ denotes Ozsv\'ath and Szab\'o's concordance homomorphism introduced in~\cite{OzsvathSzabo_03_KFHandthefourballgenus}.
\begin{Remark} The proof of Proposition~\ref{prop:betanm} uses $\Upsilon$ rather than $\HU$. It is in spirit closer to~\cite{FellerKrcatovich_16_OnCobBraidIndexAndUpsilon} (in particular, to the proof of~\cite[ Theorem~1.3]{FellerKrcatovich_16_OnCobBraidIndexAndUpsilon}), where $\Upsilon$ was used to understand cobordism distance and braid index of positive braids and $\HU$ was only discussed to make connections to the signature clearer. In contrast, the main results in this article use $\HU$, which not only makes the connection to $\fdtc$ possible, but also allows for much shorter proofs (once the formal properties of $\HU$ are established) and treatment of links (rather than just knots). However, using $\HU$ comes at the cost of no longer being able to treat some examples; in particular $\beta_{n,n}$, which realizes the braid index of its closure by Proposition~\ref{prop:betanm}, but as $\omega=n-1$ this does not follow from Theorem~\ref{thmintro:FDTCofNonminimalBraidIsBounded}. \end{Remark}
\begin{proof}[Proof of Proposition~\ref{prop:betanm}]
We show that the first singularity $t_0>0$ of $\Upsilon_K$ is strictly smaller than $\frac{2}{n-1}$. This suffices since for quasipositive knots (or more generally knots that arise as the closure of braids on which the slice-Bennequin inequality is sharp) the braid index is bounded below by $\frac{2}{t_0}$; see~\cite[Lemma~3.4 and Proposition~3.7]{FellerKrcatovich_16_OnCobBraidIndexAndUpsilon}.

Let $g$ denote the smooth $4$-ball genus $g_4(K)=\tau(K)$ of $K$ and let $L$ denote the knot obtained as the closure of the $n$-braid $\beta_{n,n}=(\delta\delta^\Delta)^{n-1}\delta$.
For all knots, the function $\Upsilon$ equals $-\tau t$ for small enough $t$; see~\cite[Proposition~1.6]{OSS_2014}.
So, we know that $\Upsilon_K(t)=-gt$ for small $t$. We will show that $\Upsilon_K(t)>-gt$ for $\frac{2}{n-1}\geq t>\frac{2(n-1)}{(n-1)^2+1}$. From this we conclude that $t_0$ is in $(0,\frac{2(n-1)}{(n-1)^2+1}]\subset(0,\frac{2}{n-1})$.

To show that the first singularity $t_0>0$ of $\Upsilon_K$ is strictly smaller than $\frac{2}{n-1}$ we use concordance properties of $\Upsilon$ established in~\cite{OSS_2014} and the following two `short' cobordisms:

\begin{Claim}\label{Claim:g4K-L}
There exists a cobordism of genus $g_4(K)-g_4(L)=g-(n-1)^2$ between $K$ and $L$. In other words,
$g_4(K\#(-L))=g-(n-1)^2$.
\end{Claim}

\begin{Claim}\label{Claim:L=diffoftorusknots}
There exists a cobordism of genus $n-2$ between \[L\et T_{n,(n-1)n+1}\#(-T_{n-1,(n-1)(n-1)+1}).\]
\end{Claim}

We postpone the proof of these claims to the end of Appendix~\ref{app:homogenizationofconcordancehomos} as they use similar ideas as the proofs there. 

Fix $t\in [\frac{2}{n},\frac{2}{n-1}]$.
Using the value of $\Upsilon_{T_{n,(n-1)n+1}}(t)$ and $\Upsilon_{T_{n-1,(n-1)(n-1)+1}}(t)$ provided in Proposition~\ref{prop:PropOfU}, we bound $\Upsilon_K$ from below as follows.

\begin{align*}
  \Upsilon_K(t)&\geq& \Upsilon_L(t)-t(g-(n-1)^2)&\\
  & \geq& \Upsilon_{T_{n,(n-1)n+1}\#(-T_{n-1,(n-1)(n-1)+1})}(t)-t(g-(n-1)^2)-t(n-2)&\\
  &=&\Upsilon_{T_{n,(n-1)n+1}}-\Upsilon_{T_{n-1,(n-1)(n-1)+1}}(t)-t(g-(n-1)^2)-t(n-2)&\\
  &=&-t(n-1)\frac{n(n-1)}{2}+(n-1)n(t-\frac{2}{n})\\&&+t(n-1)\frac{(n-1)(n-2)}{2}-t(g-(n-1)^2)-t(n-2)&\\
  &=&-t\left((n-1)\frac{n(n-1)}{2}-(n-1)\frac{(n-1)(n-2)}{2}+g-(n-1)^2\right)\\&&+(n-1)n(t-\frac{2}{n})-t(n-2)&\\
  &=&-tg+(n-1)n(t-\frac{2}{n})-t(n-2)&\\
  &=&-tg+t((n-1)n-(n-2))-(n-1)n\frac{2}{n})&\\
  &=&-tg+t((n-1)^2+1)-2(n-1),&
  \end{align*}
  where we used Claim~\ref{Claim:g4K-L} and Claim~\ref{Claim:L=diffoftorusknots} in the first and second line, respectively.
  This concludes the proof since $t((n-1)^2+1)-2(n-1)>0$ for $t>\frac{2(n-1)}{(n-1)^2+1}$ and so
  $\Upsilon_K(t)>-tg$.

\end{proof}

We now observe that the bound in Theorem~\ref{thmintro:FDTCofNonminimalBraidIsBounded} is larger than necessary for $3$-braids. This leads us to ask Question~\ref{q:n-2} in Section~\ref{sec:questions}.

\begin{prop}\label{prop:n-2} Any $3$-braid $\beta$ such that $|\omega(\beta)| > 1 = 3-2 = n-2$ realizes the braid index of its closure.
\end{prop}

\begin{proof}

We prove the contrapositive. Consider a $3$-braid $\beta$ such that the closure of $\beta$ admits a braid representative of strand number one or two. By the classification of $3$-braids in~\cite{BirmanMenasco_StudyingLinksViaClosedBraidsIII}, $\beta$ is conjugate to either $a_{1}a_{2}, a_{1}^{-1}a_{2}^{-1}, a_{1}a_{2}^{-1}$ (if it is a representative of the unknot) or $a_{1}^{k}a_{2}$ or $a_{1}^{k}a_{2}^{-1}$ for $k \in \mathbb{Z}$ (if it is a representative of a $(2,k)$ torus knot or link). One consequence of the properties listed in Proposition~\ref{fdtcproperties} (see~\cite{Malyutin_Twistnumber}, Proposition 13.1) is that if a braid $\alpha \in B_{n}$ is represented by a word containing precisely $r$ occurrences of the generator $a_{i}$ and $s$ occurrences of the generator $a_{i}^{-1}$ for $i \in \{1, \ldots, n-1\}$, then
$$ -s \leq \omega(\alpha) \leq r$$

Notice that each of the braid words listed above contains at most one negative power and at most one positive power of $a_{2}$, and hence $-1$ is a lower bound and $1$ is an upper bound for each of their fractional Dehn twist coefficients. The fractional Dehn twist coefficient is invariant under conjugation, and so this implies that $|\omega(\beta)| \leq 1$.

\end{proof}

We finish this section by showing that Theorem \ref{thmintro:FDTCofNonminimalBraidIsBounded} determines the braid index in infinitely many cases where the Morton-Franks-Williams inequality (\cite{Franks_Williams_87_BraidsAndTheJonesPolynomial}, \cite{Morton_SeifertCircles}, \cite{Morton_PolynomialsFromBraids}) is not sharp. Elrifai in \cite{Elrifai_thesis} (see also~\cite{Kawamuro_KR_MFW}) proved that for all knots and links of braid index three, the Morton-Franks-Williams inequality is sharp except for the families of knots and links which are closures of
$$ K_{k} = (a_{1}a_{2}a_{2}a_{1})^{2k} a_{1}a_{2}^{-2k-1}$$
and
$$ L_{k} = (a_{1}a_{2}a_{2}a_{1})^{2k+1} a_{1}a_{2}^{-2k+1}$$
for $k$ a positive integer.

\begin{Example}\label{Ex:elrifai_examples} The families of $3$-braids $K_{k}$ and $L_{k}$ (for $k \geq 2$ and $k \geq 1$, respectively) have fractional Dehn twist coefficients strictly larger than two (and hence have braid index three by Theorem \ref{thmintro:FDTCofNonminimalBraidIsBounded}).

To see this,
we rewrite $K_{k}$ as
$$K_{k} = (a_{1}a_{2}a_{2}a_{1})^{2k} a_{1}a_{2}^{-2k-1} = (\Delta^{2} a_{2}^{-2})^{2k} a_{1}a_{2}^{-2k-1} = (\Delta^{2})^{2k}a_{2}^{-4k}a_{1}a_{2}^{-2k-1}$$
and similarly $L_{k}$ as
$$L_{k} = (a_{1}a_{2}a_{2}a_{1})^{2k+1} a_{1}a_{2}^{-2k+1} = (\Delta^{2})^{2k+1}a_{2}^{-4k-2}a_{1}a_{2}^{-2k+1}.$$
Again using Proposition 13.1 from \cite{Malyutin_Twistnumber}, we see that $\omega(a_{2}^{-4k}a_{1}a_{2}^{-2k-1}) = 0 = \omega(a_{2}^{-4k-2}a_{1}a_{2}^{-2k+1})$ as each has only positive powers of $a_{1}$ and only negative powers of $a_{2}$. Finally, by Property (c) from Proposition \ref{fdtcproperties}, we can conclude that $\omega(K_{k}) = 2k$ and $\omega(L_{k}) = 2k+1$.
\end{Example}

\section{Questions}\label{sec:questions}
  By Theorem~\ref{thmintro:FDTCviaUpsilon}, we know more about $\HU$ than that it is Lipschitz continous (Proposition~\ref{prop:HUisLipschitz}): $\HU$ is piecewise linear with rational slopes on $[0,\frac{2}{n-1}]$ by Theorem~\ref{thmintro:FDTCviaUpsilon} and Corollary~\ref{cor:upsilonrational}. This brings us to ask:

\begin{question}\label{q:HUisPL}
Is $\HU$ piecewise linear for all $n$-braids and are all the slopes rational?
\end{question}
We remark that 
a positive answer to the following question about the $\Upsilon$-invariant of closures of braids of a fixed number of strands would imply that $\HU$ is always piecewise linear.
\begin{question}Fix an integer $n\geq 2$.
Given a knot $K$ that arises as the closure of an $n$-braid, denote by $S$ the subset of $[0,1]$ on which the piecewise-linear function $\Upsilon_K(t)$ is not smooth. Is $S$ contained in
\[\left\{\frac{2p}{q}\;|\;\text{where $p$ and $q$ are positive integers such that }|q|\leq n\right\}?\]
\end{question}

Finally, Proposition~\ref{prop:n-2} motivates the following:

\begin{Question}\label{q:n-2} Is it true that for any $n$-braid $\beta$ such that $|\omega(\beta)| > n-2$, $\beta$ realizes the braid index of its closure?
\end{Question}

Example~\ref{optimalityexample} shows that this would be the lowest possible bound for $\omega(\beta)$, as $\beta_{n,n-1}$ is an $n$-braid that does not realize the braid index of its closure and $\omega(\beta_{n,n-1}) = n-2$. Note that just as Theorem~\ref{thmintro:FDTCofNonminimalBraidIsBounded} implied Corollary~\ref{corintro:ConjMalyutinNetsetaev}, a positive answer to Question \ref{q:n-2} would imply that if an $n$-braid $\beta$ satisfies $\Delta^{2(n-1)}\preceq\beta$ or $\beta\preceq\Delta^{-2(n-1)}$, then the closure of $\beta$ does not arise as the closure of a braid on $n-1$ or fewer strands.

\appendix
\section{Homogenization of concordance homomorphisms}\label{app:homogenizationofconcordancehomos}
In this section, we establish some basic properties of the homogenizations of concordance homomorphisms. 
These properties seem to have not been established in the literature so far, although some have been claimed without proof in~\cite{FellerKrcatovich_16_OnCobBraidIndexAndUpsilon} and so we provide proofs for completeness.
Our proofs are in spirit close to the constructions of Baader~\cite{Baader_07_AsymptRasmussenInv} and Brandenbursky~\cite{Brandenbursky_11}. More concretely, the proofs of Lemma~\ref{lemma:PropofHI} and Proposition~\ref{prop:HI(alpha)-HI(beta)} are based on the following fundamental observation: given two $n$-braids $\alpha$ and $\beta$, the closure of $\alpha\beta$ and the connected sum of the closures of $\alpha$ and $\beta$ are related by a connected cobordism of Euler characteristic $n-1$. The proof of Lemma~\ref{lemma:HIwhenSBIissharp} is a variation of Rudolph's proof for the slice-Bennequin inequality given in~\cite[Lemma~4]{rudolph_QPasobstruction}.

In the entire section, $I$ denotes a knot invariant that 
descends to a homomorphism $I\colon \mathcal{C}\to \R$ with $|I(K)|\leq t_I g_4(K)$ for all knots $K$ and some real constant~$t_I$. Here $\mathcal{C}$ denotes the \emph{concordance group}---knots up to concordance with group operation given by connected sum (denoted by $\#$); in particular, for every knot $K$, the knot given as the mirror image of $K$ with reversed orientation (denoted by $-K$) represents the class of the inverse of the class of $K$.

Fix a positive integer $n$ and, for every $n$-braid $\beta$, choose an $n$-braid $\varepsilon_{\beta}$ of bounded (independent of $\beta$) length such that the closure of $\beta\varepsilon_{\beta}$ is a knot. In fact, $\varepsilon_{\beta}$ can be chosen to be of length equal to the number of components of the closure of $\beta$; in particular, of length at most $n-1$.
Brandenbursky~\cite{Brandenbursky_11} showed that there is a well-defined (independent of the choices for $\varepsilon_{\beta}$) map
\[\HI\colon B_n\to \R, \beta\mapsto \lim_{k\to\infty}\frac{I\left(\widehat{\beta^k\varepsilon_{\beta^k}}\right)}{k},\]
called the \emph{homogenization} of $I$.
In fact, $\HI$ is a homogeneous quasimorphism.
Here a \emph{quasimorphism} on a group $G$ is any map $\phi\colon G\to \R$ such that \[\sup_{(a,b)\in G\times G}|\phi(ab)-\phi(a)-\phi(b)|<\infty\] and a quasimorphism $\phi\colon G\to \R$ is called \emph{homogeneous} if $\phi(a^k)=k\phi(a)$ for all integers $k$ and all $a\in G$.
Homogeneity of $\HI$ is immediate from the construction. All homogeneous quasimorphisms are constant on conjugacy classes. We summarize properties specific to $\HI$ that get used in the main part of this text.

\begin{lemma}\label{lemma:PropofHI}
\begin{enumerate}[I)]
  \item\label{item:HI-I} If a knot $K$ is the closure of an $n$-braid $\beta$, then $|\HI(\beta)-I(K)|\leq t_I\frac{n-1}{2}$.
  \item\label{item:HI(betaa_i)}
  For all $n$-braids $\beta$, $\left|\HI(\beta a_i^{\pm1})-\HI(\beta)\right|\leq \frac{t_I}{2}$.

  \item\label{item:HIhasDefectt_I(n-1)} For all $n$-braids $\alpha$ and $\beta$, $\left|\HI({\alpha\beta})-\HI(\alpha)-\HI(\beta)\right|\leq t_I(n-1)$. If $\alpha$ and $\beta$ commute; for example, if $\alpha$ is the $n$-stranded full twist $\Delta^2$ or a power of $\beta$, then
      $\HI({\alpha\beta})=\HI(\alpha)+\HI(\beta)$.
  \item \label{item:HIofUnions}Fix positive integers $n$, $n_1$, $\cdots$, $n_l$ such that $n=\sum_{i=1}^l n_i$. If an $n$-braid $\beta$ is given as the disjoint union of braids $\beta_1$, $\cdots$ ,$\beta_l$ on $n_1$, $\cdots$, $n_l$ strands, respectively, then $\HI(\beta)=\sum_{i=1}^l \HI({\beta_i})$.
\end{enumerate}
\end{lemma}
\begin{Remark}\label{Rmk:Brandenburskysproof}
Brandenbursky proved that $\HI$ is a homogeneous quasimorphism by showing that it arises as the homogenization of the quasimorphism given by $\beta\mapsto I\left(\widehat{\beta\varepsilon_{\beta}}\right)$, which is a quasimorphism of defect at most $3t_In$.\footnote{For this, we recall that the homogenization of a quasimorphism is well-defined, which can for example be seen by adding or subtracting the defect, turning the quasimorphism into an subadditive or superadditive function, and then applying Fekete's Lemma, which states the following: given a subadditive (superadditive) real-valued sequence $(a_k)_{k=1}^{\infty}$, the limit $\lim_{k\to\infty}\frac{a_k}{k}$ exists in $[-\infty,\infty)$ ($(-\infty,\infty]$)~\cite{Fekete_23}.} A priori, this only allows to conclude that $\HI$ is a homogeneous quasimorphism of defect at most $6t_In$ rather than $t_I(n-1)$. We do not know of an example of an $I$ that shows that the bound $t_I(n-1)$ is realized. However, we do provide examples that show that the other inequalities in Lemma~\ref{lemma:PropofHI} cannot be improved; see Example~\ref{Ex:optofPropofHI}.

\end{Remark}
For the proofs we will build cobordisms between closures of braids and then apply the fact that if there is a cobordism of genus $g$ between two knots $K_1$ and $K_2$, then
\begin{equation}\label{eq:I(K)-I(L)<tg}|I(K_1)-I(K_2)|\leq t_I g,\end{equation}
since $I(K_1)-I(K_2)=I(K_1\# -K_2)$ and $g_4(K_1\# -K_2)\leq g$. Here a cobordism $C$ between two links $L_0$ and $L_1$ is a smooth oriented surface in $S^3\times [0,1]$ such that $\partial C=L_0\times\{0\} \cup L_1\times\{1\}$. 
\begin{proof}[Proof of Lemma~\ref{lemma:PropofHI}
]
\ref{item:HI-I}):
For every fixed $k$,
we claim that there exists a cobordism of genus $\frac{(n-1)(k-1)+\l}{2}$ between $kK$ and $\widehat{\beta^k\varepsilon_{\beta^k}}$,
where $kK$ denotes the $k$-fold connect sum $K\# \cdots \# K$ and $\varepsilon_{\beta^k}$ is an $n$-braid of length $\l$ at most $n-1$ such that the closure of $\beta^k\varepsilon_{\beta^k}$ is a knot.
In fact, there exists a cobordism $C$ between $kK$ and $\widehat{\beta^k\varepsilon_{\beta^k}}$ given by $(k-1)(n-1)+\l$ band moves; see Figure~\ref{fig:cob}.

\begin{figure}[h]
\centering
\includegraphics[width=12cm]{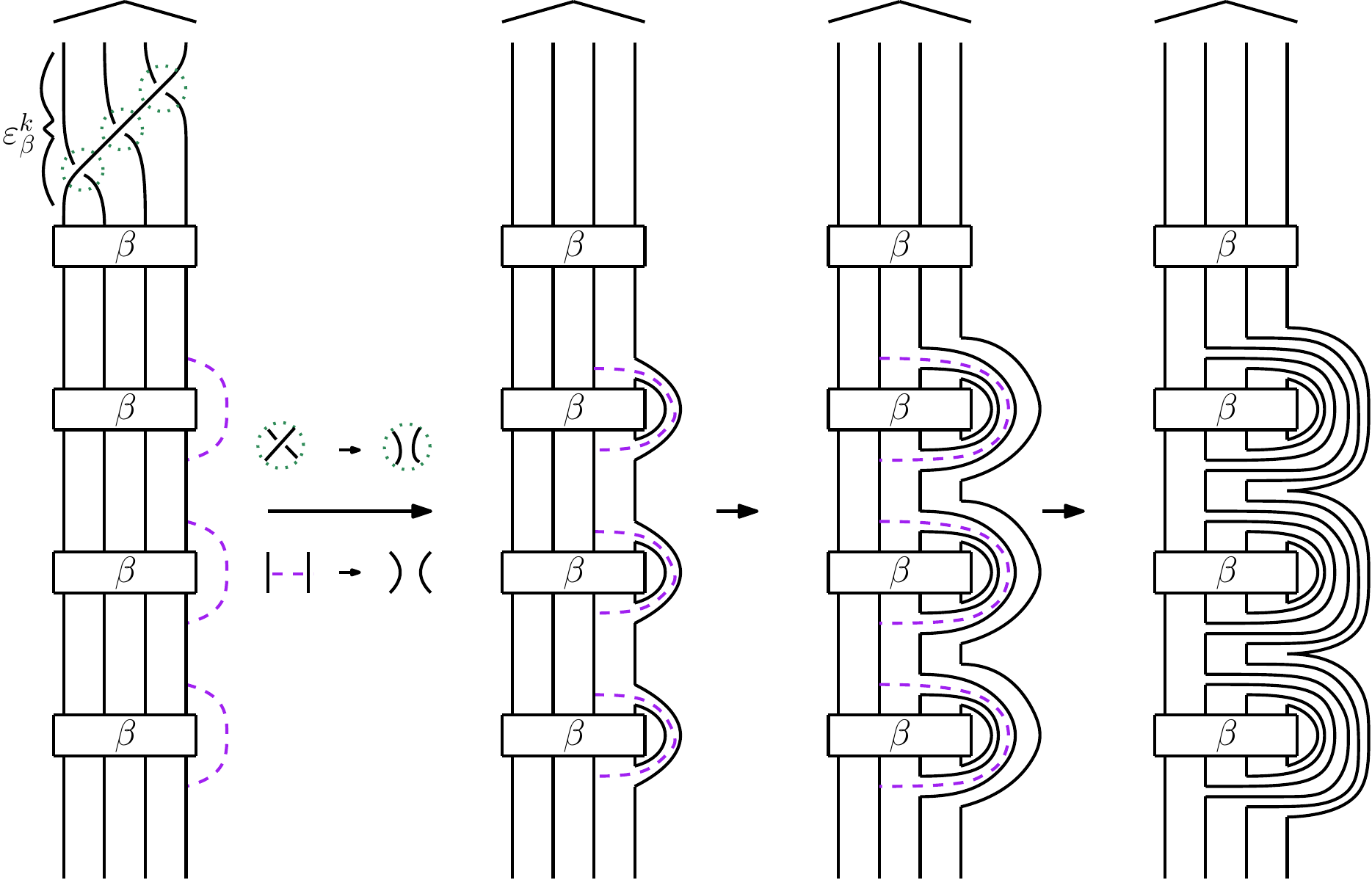}
\caption{Illustration of the
band moves (and crossing resolutions, which are also band moves) that yield the cobordism $C$ from $\widehat{\beta\varepsilon_{\beta^k}}$ to $kK$ for $n=4$, $k=4$, and $\epsilon_{\beta^k}=a_1a_2a_3$. }
\label{fig:cob}
\end{figure}

In particular, the cobordism $C$ has Euler characteristic $-(k-1)(n-1)-\l$, is connected, and has two boundary components; thus, its genus is
$\frac{(k-1)(n-1)+\l}{2}$.

We calculate
\begin{align*}\left|I\left(\widehat{\beta^k\varepsilon_k}\right)-kI(K)\right|
&=\left|I\left(\widehat{\beta^k\varepsilon_k}\right)-I(kK)\right|
\overset{\text{\eqref{eq:I(K)-I(L)<tg}}}{\leq} t_I\frac{(n-1)(k-1)+\l}{2}.\end{align*}
Dividing by $k$ and taking the limit $k\to\infty$ yields
\[\left|\widetilde{I}(\beta)-I(K)\right|=\left|\lim_{k\to\infty}\frac{I\left(\widehat{\beta^k\varepsilon_{\beta^k}}\right)}{k}-I(K)\right|
=\left|\lim_{k\to\infty}\frac{I\left(\widehat{\beta^k\varepsilon_{\beta^k}}\right)-kI(K)}{k}\right|\leq
\frac{t_I(n-1)}{2}.\]

\ref{item:HI(betaa_i)}): For every $k$, suppose $\epsilon_{\beta^k}$ and $\epsilon_{(\beta a_i^{\pm1})^k}$ are braids of length $\l\leq n-1$ and $\l'\leq n-1$, respectively, such that the closures of
$\beta^k\epsilon_{\beta^k}$ and $(\beta a_i^{\pm1})^k\epsilon_{(\beta a_i^{\pm1})^k}$ are knots. If braids $\alpha$ and $\alpha'$ differ by adding or removing a generator $a_i$ or $a_i^{-1}$, then their closures are related by a cobordism of Euler characteristic $-1$. Indeed, as discussed in Figure~\ref{fig:cob}, adding or removing a crossing can be realized by a cobordism consisting of one 1-handle.
So, since the braid $\beta^k\epsilon_{\beta^k}$ can be turned into the braid $(\beta a_i^{\pm1})^k\epsilon_{(\beta a_i^{\pm1})^k}$ by removing $\l$ + k generators and adding $\l'$ generators, there exists a cobordism of Euler characteristic $-\l-\l'-k$, i.e.~genus $\frac{\l+\l'+k}{2}$, between the knots given as the closure of $\beta^k\epsilon_{\beta^k}$ and $(\beta a_i^{\pm1})^k\epsilon_{(\beta a_i^{\pm1})^k}$. Consequently, we have \[\left|I\left(\widehat{(\beta a_i^{\pm1})^k\epsilon_{(\beta a_i^{\pm1})^k}}\right)-I\left(\widehat{\beta^k\epsilon_{\beta^k}}\right)\right|
\overset{\eqref{eq:I(K)-I(L)<tg}}{\leq} t_I\frac{\l+\l'+k}{2}.\]
Dividing by $k$ and taking the limit $k\to\infty$ yields $\left|\HI(\beta a_i)-\HI(\beta)\right|\leq \frac{t_I}{2}$.


\ref{item:HIhasDefectt_I(n-1)}): Fix a positive integer $k$, and let $\epsilon_{(\alpha\beta)^k}$, $\epsilon_{\alpha^k}$, and $\epsilon_{\beta^k}$ denote $n$-braids of length $\l,\l_\alpha,\l_\beta\leq n-1$, respectively, such that closures of $(\alpha\beta)^k\epsilon_{(\alpha\beta)^k}$, $\alpha^k\epsilon_{\alpha^k}$, and $\beta^k\epsilon_{\beta^k}$ are knots.

We first observe that there exists a cobordism of Euler characteristic $-(n-1)k-\l-\l_\alpha$
between the knot $\widehat{(\alpha\beta)^k\epsilon_{(\alpha\beta)^k}}$ and the link
$\widehat{\alpha^k\epsilon_{\alpha^k}}\# k\widehat{\beta}$, where $k\widehat{\beta}$ denotes the connected sum of $k$ copies of $\widehat{\beta}$ with the summing operation happening along the component of $\widehat{\beta}$ that contains the strand of $\beta$ that ends left-most on the top of $\beta$; see Figure \ref{fig:connectsum}.
\begin{figure}[h]
\centering
\includegraphics[scale=0.6, angle=90]{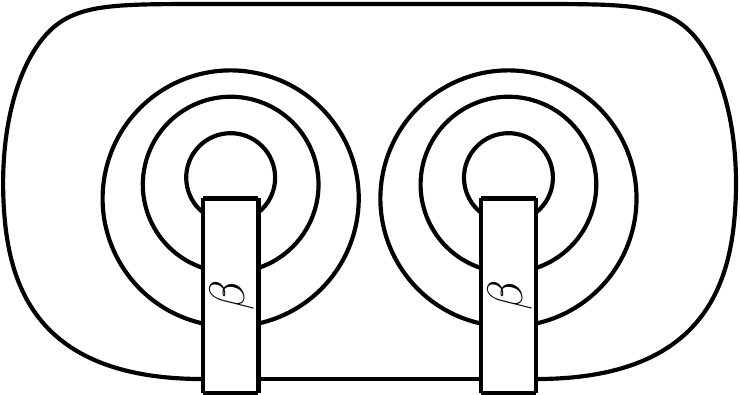}
\caption{An illustration of $2\widehat{\beta}$ for $n=4$.}
\label{fig:connectsum}
\end{figure}
Figure~\ref{fig:cob2} shows how such a cobordism is given by band moves.
\begin{figure}[h]
\centering
\includegraphics[width=12cm]{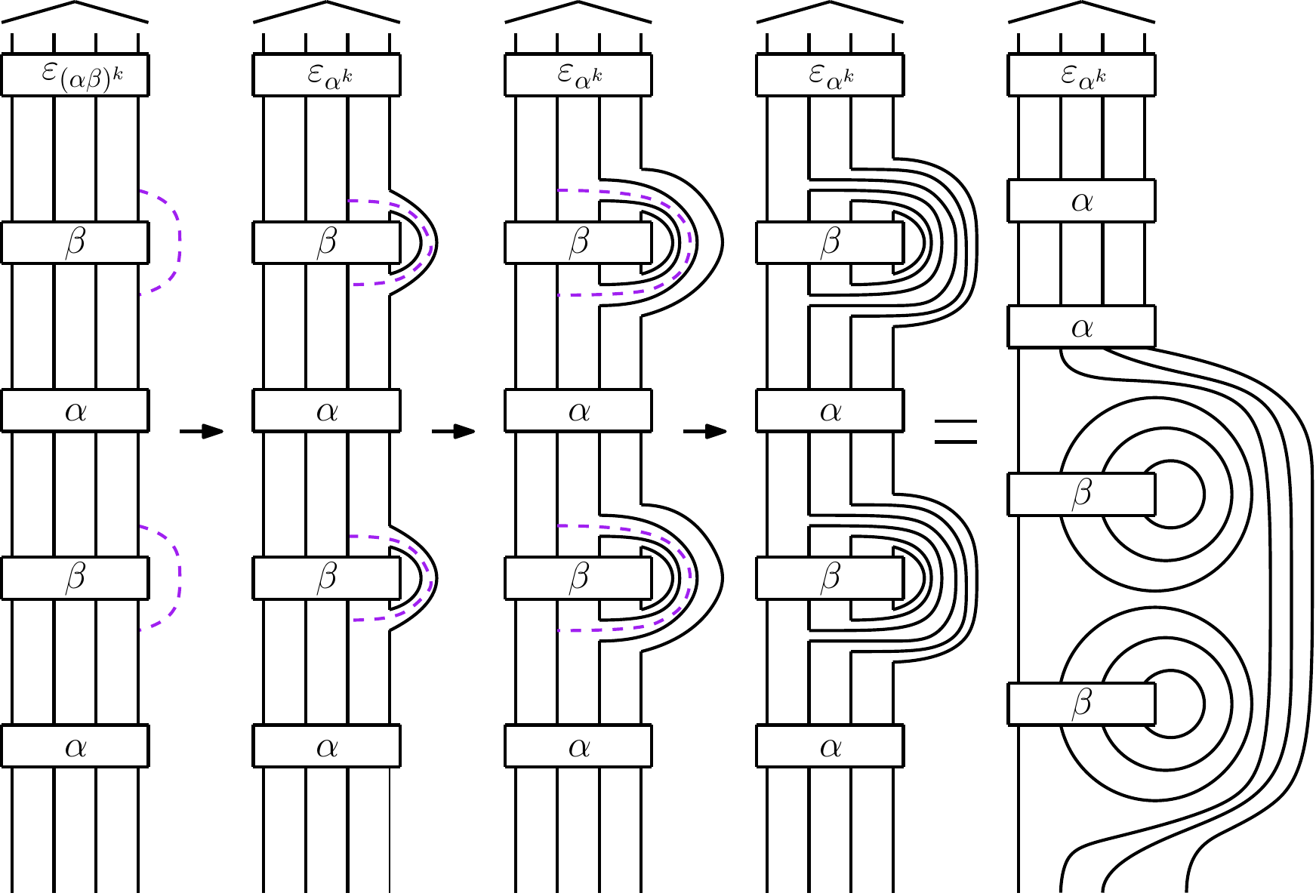}
\caption{Illustration of the
band moves (and crossing resolutions, which are also band moves) that yield a cobordism from the knot $\widehat{(\alpha\beta)^k\epsilon_{(\alpha\beta)^k}}$ to the link
$\widehat{\alpha^k\epsilon_{\alpha^k}}\# k\widehat{\beta}$ for $n=4$, $k=2$. Arrows indicate saddle moves, which for the first arrow includes crossing resolutions and addition (compare with Figure~\ref{fig:cob}). The equality to the right is an isotopy of links. The latter can be seen by recalling that since $\widehat{\alpha^k\epsilon_{\alpha^k}}$ is a knot, the connected sum with several $\widehat{\beta}$ yields isotopic links independent of where on $\widehat{\alpha^k\epsilon_{\alpha^k}}$ the $\widehat{\beta}$ get added.}
\label{fig:cob2}
\end{figure}

By the same argument as in the proof of~\ref{item:HI-I} (see Figure~\ref{fig:cob}), there exists a cobordism of Euler characteristic $-(n-1)(k-1)-\l_\beta$ between $k\widehat{\beta}$ and $\widehat{\beta^k\epsilon_{\beta^k}}$.

These two cobordisms may be concatenated to yield a cobordism of genus
\[\frac{(n-1)(2k-1)+\l+\l_\alpha+\l_\beta}{2}\] between the knots
$\widehat{(\alpha\beta)^k\epsilon_{(\alpha\beta)^k}}$ and $\widehat{\alpha^k\epsilon_{\alpha^k}}\#\widehat{\beta^k\epsilon_{\beta^k}}$.
Therefore, we have
\[\left|I\left(\widehat{(\alpha\beta)^k\epsilon_{(\alpha\beta)^k}}\right)- I\left(\widehat{\alpha^k\epsilon_{\alpha^k}}\#\widehat{\beta^k\epsilon_{\beta^k}}\right)\right|\overset{\eqref{eq:I(K)-I(L)<tg}}{\leq}
t_I\frac{(n-1)k+\l+\l_\alpha+(n-1)k+\l_\beta}{2}.\]
Dividing by $k$ and taking the limit $k\to \infty$ implies $\left|\HI({\alpha\beta})-\HI(\alpha)-\HI(\beta)\right|\leq t_I(n-1)$.

\

In the case when $\alpha$ and $\beta$ commute,
we use the fact that
\[\lim_{k\to\infty}\frac{|\HI(\alpha^k\beta^k)-\HI(\alpha^k)-\HI(\beta^k)|}{k}=0\]
to conclude
\[\HI(\alpha\beta)=\lim_{k\to\infty}\frac{\HI(\alpha^k\beta^k)}{k}=\lim_{k\to\infty}\frac{\HI(\alpha^k)}{k}+\lim_{k\to\infty}\frac{\HI(\beta^k)}{k}
=\HI(\alpha)+\HI(\beta),\]
where $(\alpha\beta)^k=(\alpha)^k(\beta)^k$ was used in the first equality.

\

\ref{item:HIofUnions}):
For any $n_i$-braids $\beta_i$, let $\beta$ denote their disjoint union. In order to give a braid word for $\beta$, we must shift each $\beta_{i}$ by the appropriate number of strands. Indeed, if we let $\beta'_i$ denote the $n$-braid obtained from a braid word for $\beta_i$ by replacing $a_k^{\pm1}$ by $a_{k+\sum_{j<i}n_j}^{\pm1}$, then $\beta=\beta'_1\beta'_2\cdots\beta'_l$. We note that $\beta'_i$ and $\beta'_j$ commute for all $i,j\leq l$. Therefore, $\HI(\beta)=\sum_{i=1}^l \HI({\beta'_i})$ by~\ref{item:HIhasDefectt_I(n-1)}. Thus, we are left with showing
$\HI({\beta'_i})=\HI({\beta_i})$ for all $i\leq l$. We approach this by observing that, while $\beta'_i$ and $\beta_i$ have different numbers of strands, since the braided portions of $\beta'_i$ and $\beta_i$ are identical we may choose convenient $\epsilon$'s to concatenate with ${\beta'_i}^k$ and ${\beta_i}^k$ in the computation of $\HI$ in order to make the closures be isotopic knots.

Fix a positive integer $k$ and let $\epsilon_{\beta_i^k}$ be an $n_i$-braid of length $\l_i$ such that the closure of $\beta_i^k\epsilon_{\beta_i^k}$ is a knot.
Let $\epsilon'_{\beta_i^k}$ be the $n$-braid obtained from a braid word for $\epsilon_{\beta_i^k}$ by replacing $a_k^{\pm1}$ by $a_{k+\sum_{j<i}n_j}^{\pm1}$ and set
\[\epsilon_{{\beta'}_i^k}=a_1a_2\cdots a_{\sum_{j<i}n_j}\epsilon'_{\beta_i^k}.\]
We note that ${\beta'}_i^k\epsilon_{{\beta'}_i^k}$ and ${\beta_i}^k\epsilon_{{\beta_i}^k}$ have isotopic closures and so we conclude
\[\HI(\beta_i)=\lim_{k\to\infty}\frac{I\left(\widehat{\beta_i^k\epsilon_{\beta_i^k}}\right)}{k}
=\lim_{k\to\infty}\frac{I\left(\widehat{{\beta'}_i^k\epsilon_{{\beta'}_i^k}}\right)}{k}=\HI(\beta'_i).\]
\end{proof}

Two famous examples for $I$, Ozsv\'ath and Szab\'o's $\tau$ invariant and Rassmusen's $s$ invariant, turn out to have very simple homogenizations. The following implies this and can be seen as a version of~\cite[Theorem~3.5]{Brandenbursky_11} that depends on the braid index. For coprime positive integers, we denote by $T_{p,q}$ the torus knot given as the closure of the $p$-braid $(a_1a_2\cdots a_{p-1})^q$.

\begin{lemma}\label{lemma:HIwhenSBIissharp}
Fix an integer $n\geq 2.$
If $I(T_{n,nk+1})=t_Ig(T_{n,nk+1})=\frac{(n-1)nk}{2}$
for all positive integers $k$, then
\[\HI(\beta)=t_I\frac{\wr(\beta)}{2}\text{ for all $n$-braids }\beta.\]
\end{lemma}
\begin{proof}
The equality $I(T_{n,nk+1})=t_Ig_4(T_{n,nk+1})$, 
for all positive integers $k$, implies that the slice-Bennequin inequality holds for all $n$-braids; that means,
if the knot $K$ is the closure of an $n$-braid $\alpha$, then
\begin{equation}\label{eq:sB}t_I\frac{\wr(\alpha)-(n-1)}{2}\leq I(K)\leq t_I\frac{\wr(\alpha)+n-1}{2}.\end{equation}
For completeness, we provide a proof of~\eqref{eq:sB} following~\cite[Lemma~4]{rudolph_QPasobstruction}; compare also with the proof of~\cite[Corollary~11]{Livingston_Comp}. We only establish the first inequality of~\eqref{eq:sB} as the second one follows by applying the first to $-\alpha$. Removing all $a_i^{-1}$ in a braid word for $\alpha$, then adding generators $a_i$ allows us to turn $\alpha$ into $(a_1a_2\cdots a_{n-1})^{nk+1}$ for some positive integer $k$, which we additionally chose such that $(n-1)(nk+1)-\wr(\alpha)>0$. Since adding and removing generators yields a cobordism consisting of a $1$-handle between the corresponding closures, this implies that there exists a cobordism of Euler characteristic $-(n-1)(nk+1)+\wr(\alpha)$ between $K$ and $T_{n,nk+1}$.
Thus, we find
\begin{align*}t_I\frac{(n-1)nk}{2}-I(K)&=I(T_{n,nk+1})-I(K)
\\&\overset{\text{\eqref{eq:I(K)-I(L)<tg}}}{\leq} t_I\frac{(n-1)(nk+1)-\wr(\alpha)}{2}\\&=t_I\frac{(n-1)nk}{2}-t_I\frac{\wr(\alpha)-(n-1)}{2},\end{align*} as wanted,
where the assumption on the value of torus knots was used in the first line.

The statement of the lemma follows from~\eqref{eq:sB} by setting $\alpha=\beta^k\varepsilon_k$ and $K=\widehat{\beta^k\varepsilon_k}$, dividing by $k$ and taking the limit $k\to\infty$. 
\end{proof}

The ideas of the proof of Lemma~\ref{lemma:PropofHI} can be used to establish the following.
\begin{prop}\label{prop:HI(alpha)-HI(beta)}
Fix positive integers $n$ and $m$. If an $n$-braid $\beta$ and an $m$-braid $\alpha$ have isotopic links as their closure (or, more generally, concordant links as their closure), then
\[\left|\HI(\beta)-\HI(\alpha)\right|\leq t_I\frac{n-1+m-1}{2}.\]
\end{prop}
We use Proposition~\ref{prop:HI(alpha)-HI(beta)} crucially in the proof of Theorem~\ref{thmintro:FDTCofNonminimalBraidIsBounded}. In Example~\ref{Ex:optofPropofHI}, we comment on the optimality of Proposition~\ref{prop:HI(alpha)-HI(beta)}.

\begin{proof}[Proof of Proposition~\ref{prop:HI(alpha)-HI(beta)}]
We will first prove the statement in the case that the closure of one (and thus both) of $\alpha$ and $\beta$ are knots.

Fix a positive integer $k$ and let $\varepsilon_{\alpha^k}$ and $\varepsilon_{\beta^k}$ be braids given by braid words of length $\l\leq n-1$ and $\l'\leq m-1$ such that $\beta^k\varepsilon_{\beta^k}$ and $\alpha^k\varepsilon_{\alpha^k}$ are braids with closures that are knots. We claim that there exists a cobordism of genus $\frac{(n-1+m-1)(k-1)+\l+\l'}{2}$ between the closures of $\beta^k\varepsilon_{\beta^k}$ and $\alpha^k\varepsilon_{\alpha^k}$.
To see this, let $k\widehat{\beta}$ denote the knot obtained as the connect sum of $k$ copies the knot $\widehat{\beta}$.
By the argument given in the proof of~\ref{item:HI-I} of Lemma~\ref{lemma:PropofHI}, there exists a cobordism $C$ from the closure of $\beta^k\varepsilon_{\beta^k}$ to $k\widehat{\beta}$ given by $(k-1)(n-1)+\l$ band moves; see Figure~\ref{fig:cob}.
Similarly, there is a cobordism $D$ from $k\widehat{\alpha}$---the connect sum of $k$ times the closure of $\alpha$---to the closure of $\alpha^k\varepsilon_{\alpha^k}$ given by $(k-1)(m-1)+\l'$ band moves. Note that  knots $k\widehat{\beta}$ and $k\widehat{\alpha}$ are concordant, say by a concordance $A$, since the knots $\widehat{\beta}$ and $\widehat{\alpha}$ are concordant by assumption.
The concatenation of the cobordisms $C$, $A$, and $D$ yield a cobordism between the closure of $\beta^k\varepsilon_{\beta^k}$ and $\alpha^k\varepsilon_{\alpha^k}$ with genus $\frac{(k-1)(n-1)+\l+(k-1)(m-1)+\l'}{2}\leq k\frac{n-1+m-1}{2}$.

The statement follows from the existence of the above cobordism by the following calculation:
\begin{align*}\left|\HI(\beta)-\HI(\alpha)\right|
&=\left|\lim_{k\to\infty}\frac{I\left({\widehat{\beta^k\varepsilon_{\beta^k}}}\right)}{k}
-\lim_{k\to\infty}\frac{I\left({\widehat{\alpha^k\varepsilon_{\alpha^k}}}\right)}{k}\right|\\
&=\left|\lim_{k\to\infty}\frac{I\left(\widehat{\beta^k\varepsilon_{\beta^k}}\right)
-I\left(\widehat{\alpha^k\varepsilon_{\alpha^k}}\right)}{k}\right|\\
&\overset{\eqref{eq:I(K)-I(L)<tg}}{\leq}\lim_{k\to\infty}\frac{t_Ik\frac{n-1+m-1}{2}}{k}=t_I\frac{n-1+m-1}{2}.
\end{align*}

\

It remains to discuss the case when the closure of $\beta$ (and thus $\alpha$) is a link with several components.
The argument remains the same as above, but it only works for a particular choice of links $k\widehat{\beta}$ and $k\widehat{\alpha}$, which we elaborate below.

Let $C$ be a concordance (a union of annuli smoothly embedded in $S^3\times[0,1]$) between $\widehat{\beta}$ and $\widehat{\alpha}$. The concordance $C$ induces a bijection between the connected components of the links $\widehat{\beta}$ and $\widehat{\alpha}$: connected components of the links are related if they are contained in the same subannulus of $C$. We pick $i\leq m-1$ such that under this bijection the connected component of $\widehat{\beta}$ that contains the strand that ends left-most on the top of $\beta$ gets map to the connected component of $\widehat{\alpha}$ that contains the strand of $\alpha$ that ends $i$th on the top of $\alpha$. For example, let $\beta$ be the $3$-braid $a_2^3$ and $\alpha$ be the $3$-braid $a_1^3$. The closure of both of these are an unknot disjoint union a $T_{2,3}$ (i.e.~a trefoil). A concordance has to relate the two unknots, so in this case $i=3$.
We may conjugate $\alpha$ by a braid $\gamma$ such that the strand of $\gamma\alpha\gamma^{-1}$ that ends left-most on the top contains the strand of $\alpha$ that ends $i$th on the top of $\alpha$. For the above $3$-braid example, where $\beta=a_2^3$ and $\alpha=a_1^3$, we could choose $\gamma=a_1a_2$.

We now prove the statement for the braids $\beta$ and $\alpha$, where, without loss of generality (by the previous paragraph and the fact that conjugated braids have the same closure), we may and do assume the following: the bijection induced by the concordance $C$ relates the connected component of $\widehat{\beta}$ that contains the strand that ends left-most on the top of $\beta$ to the
connected component of $\widehat{\alpha}$ that contains the strands that ends left-most on the top of $\alpha$.

As in the proof of Lemma~\ref{lemma:PropofHI}.\ref{item:HIhasDefectt_I(n-1)}, we choose $k\widehat{\beta}$ to be the connected sum of $k$ times the link $\widehat{\beta}$, where the connected sum is done along the connected components of $\widehat{\beta}$ that contains the strand that ends left-most at the top of $\beta$; see~Figure~\ref{fig:connectsum}. Similarly, we set $k\widehat{\alpha}$ to be the connected sum of $k$ times the link $\widehat{\alpha}$, where the connected sum is done along the connected components of $\widehat{\alpha}$ that contains the strand that ends left-most on top of $\alpha$. The assumption made in the last paragraph guarantees that the links $k\widehat{\beta}$ are concordant to $k\widehat{\alpha}$: indeed, take $A$ to be the concordance (a cobordism that is a union of $k$ annuli) that is given by the concordance $C$ on the $k$ summands of $k\widehat{\beta}$ and $k\widehat{\alpha}$. With this set-up, we conclude the proof as in the case where $\widehat{\beta}$ and $\widehat{\alpha}$ are knots.
\end{proof}

\subsection*{Optimality of the inequalities in Lemma~\ref{lemma:PropofHI} and Proposition~\ref{prop:HI(alpha)-HI(beta)}}
We provide examples that show that, in general, the inequalities in Lemma~\ref{lemma:PropofHI}.\ref{item:HI-I}$\&$\ref{item:HI(betaa_i)} and Proposition~\ref{prop:HI(alpha)-HI(beta)} cannot be improved.
\begin{Example}\label{Ex:optofPropofHI}
Let $I$ be a concordance homomorphism such that $t_I=1$ and $I$ satisfies the assumption of Lemma~\ref{lemma:HIwhenSBIissharp} for all $n\geq 2$; for example, take $I$ to be Ozsv\'ath and Szab\'o's $\tau$ invariant.

Fix a positive integer $n\geq 2$.
The $n$-braid $\beta=a_1a_2\cdots a_{n-1}$ has $\HI(\beta)=\frac{\wr(\beta)}{2}=\frac{n-1}{2}$, while its closure $K$ has $I(K)=0$ since it is the unknot. Thus, we have equality
\[|\HI(\beta)-I(K)|=\frac{n-1}{2}= t_I\frac{n-1}{2}\] in Lemma~\ref{lemma:PropofHI}.\ref{item:HI-I}.
Taking $\beta$ to be the trivial $n$-braid yields equality in Lemma~\ref{lemma:PropofHI}.\ref{item:HI(betaa_i)}.

For positive integers $n$ and $m$, we set
\[\beta=a_1a_2\cdots a_{n-1}\et\alpha=(a_1a_2\cdots a_{m-1})^{-1}.\]
We have
\[\HI(\beta)=\frac{\wr(\beta)}{2}=\frac{n-1}{2}\et \HI(\alpha)=\frac{\wr(\alpha)}{2}=\frac{-m+1}{2},\]
which yields that $\beta$ and $\alpha$ are braids with isotopic closure such that the inequality in Proposition~\ref{prop:HI(alpha)-HI(beta)} is an equality:
\[\left|\HI(\beta)-\HI(\alpha)\right|=\frac{n-1+m-1}{2}=t_I\frac{n-1+m-1}{2}.\]
\end{Example}

\subsection*{Proofs of Claim~\ref{Claim:g4K-L} and Claim~\ref{Claim:L=diffoftorusknots}}

We conclude this paper by providing the proofs for Claim~\ref{Claim:g4K-L} and Claim~\ref{Claim:L=diffoftorusknots}.

\begin{proof}[Proof of Claim~\ref{Claim:g4K-L}]
  The knot $K$ is the closure of
  \[\beta=\alpha_1\delta\beta_1\delta^\Delta\alpha_2\delta\beta_2\cdots\alpha_{n-1}\delta\beta_{n-1}\delta^\Delta\alpha_{n}\delta\beta_{n},\]
      where the $\alpha_j$ and $\beta_j$ are (possibly trivial) quasipositive $n$-braids.
      We note that the $n$-braid $\beta$ can be changed into the $n$-braid $\beta_{n,n}=(\delta\delta^\Delta)^{n-1}\delta$ by removing
      \[\wr(\beta)-\wr(\beta_{n,n})=\sum_{j=1}^{n} \wr{\alpha_j}+\sum_{j=1}^{n} \wr{\beta_j}\]
      positive generators $a_i$ in a braid word for $\beta$. For this, we recall that the $\alpha_j$ and $\beta_j$ are given by braid words that are products of the form $\omega a_i\omega^{-1}$ and so removing the middle $a_i$ in each such conjugate yields a braid word for $\beta_{n,n}$.
      Therefore (as in the proof of Lemma~\ref{lemma:PropofHI}), there is a cobordism between $K$ and $L=\widehat{\beta_{n,n}}$ given by
      \[\wr(\beta)-\wr(\beta_{n,n})=2g+(n-1)-((2n-1)(n-1))=2(g-(n-1)^2)\] many $1$-handles. In particular, this cobordism is connected and of genus $\frac{2(g-(n-1)^2)}{2}$, as wanted in Claim~\ref{Claim:g4K-L}.
\end{proof}

\begin{proof}[Proof of Claim~\ref{Claim:L=diffoftorusknots}]
We first observe that
\[\beta_{n,n}=(\Delta^2)^{n-1}(a_2\cdots a_{n-1})^{-(n-1)(n-1)}(a_1a_2\cdots a_{n-1}).\] Thus, $\beta_{n,n}$ is conjugate to the $n$-braid \begin{align*}
                            \beta'_{{n,n}}&=(a_1a_2\cdots a_{n-1})(\Delta^2)^{n-1}(a_2\cdots a_{n-1})^{-(n-1)(n-1)}\\&=(a_1a_2\cdots a_{n-1})^{(n-1)n+1}(a_2\cdots a_{n-1})^{-(n-1)(n-1)};\end{align*}
in particular,
$L=\widehat{\beta_{{n,n}}}=\widehat{\beta'_{{n,n}}}$.
By adding $n-2$ generators we can turn $\beta'_{{n,n}}$ into
$\beta''_{{n,n}}=(a_1a_2\cdots a_{n-1})^{(n-1)n+1}(a_2\cdots a_{n-1})^{-(n-1)(n-1)-1}.$
Consequently, there exists a cobordism $C$ between $L$ and $\widehat{\beta''_{{n,n}}}$ of Euler characteristic $-n+2$ given by $n-2$ many $1$-handles between $L$ and $\widehat{\beta''_{{n,n}}}$.
Also, $n-2$ many band moves turn the closure of $\beta''_{{n,n}}$ into the connect sum of \[T_{n,(n-1)n+1}=\widehat{(a_1a_2\cdots a_{n-1})^{(n-1)n+1}}\et -T_{n-1,(n-1)(n-1)+1}.\] This can be seen by an argument similar to the proof of Lemma~\ref{lemma:PropofHI}.\ref{item:HIhasDefectt_I(n-1)}; compare to Figure~\ref{fig:cob2}.
This gives rise to a cobordism $D$ of Euler characteristic $-n+2$ between \[\widehat{\beta''}\et T_{n,(n-1)n+1}\#(-T_{n-1,(n-1)(n-1)+1}).\]
Concatenating $C$ and $D$ yields a cobordism of genus $n-2$ between $L$ and \newline $T_{n,(n-1)n+1}\#(-T_{n-1,(n-1)(n-1)+1})$, as wanted.
\end{proof}

\bibliographystyle{alpha}
\bibliography{peterbib}

\newcommand{\etalchar}[1]{$^{#1}$}
\def\cprime{$'$}
\begin{thebibliography}{DDRW02}

\bibitem[Ale23]{Alexander_23_ALemmaOnSystemsOfKnottedCurves}
J.~W. Alexander.
\newblock A lemma on systems of knotted curves.
\newblock {\em Proc. Nat. Acad. Sci. USA}, 9:93--95, 1923.

\bibitem[Art25]{Artin_TheorieDerZoepfe}
Emil Artin.
\newblock Theorie der {Z}{\"o}pfe.
\newblock {\em Abh. Math. Sem. Univ. Hamburg}, 4(1):47--72, 1925.

\bibitem[Baa07]{Baader_07_AsymptRasmussenInv}
Sebastian Baader.
\newblock Asymptotic {R}asmussen invariant.
\newblock {\em C. R. Math. Acad. Sci. Paris}, 345(4):225--228, 2007.

\bibitem[BB05]{Birman_Brendle_braidssurvey}
Joan~S. Birman and Tara~E. Brendle.
\newblock Braids: a survey.
\newblock In {\em Handbook of knot theory}, pages 19--103. Elsevier B. V.,
  Amsterdam, 2005.

\bibitem[BE13]{BaldwinEtnyre_admissiblesurgery}
John~A. Baldwin and John~B. Etnyre.
\newblock Admissible transverse surgery does not preserve tightness.
\newblock {\em Math. Ann.}, 357(2):441--468, 2013.

\bibitem[BM93]{BirmanMenasco_StudyingLinksViaClosedBraidsIII}
Joan~S. Birman and William~W. Menasco.
\newblock Studying links via closed braids. {III}. {C}lassifying links which
  are closed {$3$}-braids.
\newblock {\em Pacific J. Math.}, 161(1):25--113, 1993.

\bibitem[BM06]{BirmanMenasco_StabilizationI}
Joan~S. Birman and William~W. Menasco.
\newblock Stabilization in the braid groups. {I}. {MTWS}.
\newblock {\em Geom. Topol.}, 10:413--540, 2006.

\bibitem[Bra11]{Brandenbursky_11}
M.~Brandenbursky.
\newblock On quasi-morphisms from knot and braid invariants.
\newblock {\em J. Knot Theory Ramifications}, 20(10):1397--1417, 2011.

\bibitem[DDRW02]{Dehornoy_WhyAreBraidsOrderable}
Patrick Dehornoy, Ivan Dynnikov, Dale Rolfsen, and Bert Wiest.
\newblock {\em Why are braids orderable?}, volume~14 of {\em Panoramas et
  Synth\`eses [Panoramas and Syntheses]}.
\newblock Soci\'et\'e Math\'ematique de France, Paris, 2002.

\bibitem[Deh94]{Dehornoy_94_Braidgroups}
Patrick Dehornoy.
\newblock Braid groups and left distributive operations.
\newblock {\em Trans. Amer. Math. Soc.}, 345(1):115--150, 1994.

\bibitem[DP13]{DynnikovPrasolov_13}
I.~A. Dynnikov and M.~V. Prasolov.
\newblock Bypasses for rectangular diagrams. {A} proof of the {J}ones
  conjecture and related questions.
\newblock {\em Trans. Moscow Math. Soc.}, 74(1):97--144, 2013.

\bibitem[Elr88]{Elrifai_thesis}
E.A. Elrifai.
\newblock {\em Positive braids and {L}orenz links}.
\newblock PhD thesis, Liverpool University, 1988.

\bibitem[EVHM15]{Etnyre_MonoidsMappingClassGroup}
John~B. Etnyre and Jeremy Van Horn-Morris.
\newblock Monoids in the mapping class group.
\newblock In {\em Interactions between low-dimensional topology and mapping
  class groups}, volume~19 of {\em Geom. Topol. Monogr.}, pages 319--365. Geom.
  Topol. Publ., Coventry, 2015.

\bibitem[Fek23]{Fekete_23}
M.~Fekete.
\newblock {\"U}ber die {V}erteilung der {W}urzeln bei gewissen algebraischen
  {G}leichungen mit ganzzahligen {K}oeffizienten.
\newblock {\em Math. Z.}, 17(1):228--249, 1923.

\bibitem[FGR{\etalchar{+}}99]{Fenn_Orderingbraidgroups}
R.~Fenn, M.~T. Greene, D.~Rolfsen, C.~Rourke, and B.~Wiest.
\newblock Ordering the braid groups.
\newblock {\em Pacific J. Math.}, 191(1):49--74, 1999.

\bibitem[FK17]{FellerKrcatovich_16_OnCobBraidIndexAndUpsilon}
Peter Feller and David Krcatovich.
\newblock On cobordisms between knots, braid index, and the upsilon-invariant.
\newblock {\em Mathematische Annalen}, 369:301--329, 2017.
\newblock ArXiv:1602.02637 [math.GT].

\bibitem[FW87]{Franks_Williams_87_BraidsAndTheJonesPolynomial}
John Franks and R.~F. Williams.
\newblock Braids and the {J}ones polynomial.
\newblock {\em Trans. Amer. Math. Soc.}, 303(1):97--108, 1987.

\bibitem[GG05]{GambaudoGhys_BraidsSignatures}
Jean-Marc Gambaudo and {\'E}tienne Ghys.
\newblock Braids and signatures.
\newblock {\em Bull. Soc. Math. France}, 133(4):541--579, 2005.

\bibitem[GLW18]{GrigsbyLicataWehrli_16}
J.~Elisenda Grigsby, Anthony~M. Licata, and Stephan~M. Wehrli.
\newblock Annular {K}hovanov homology and knotted {S}chur-{W}eyl
  representations.
\newblock {\em Compos. Math.}, 154(3):459--502, 2018.
\newblock ArXiv:1612.05953 [math.GT].

\bibitem[GO89]{Gabai_EssentialLaminations3Manifolds}
David Gabai and Ulrich Oertel.
\newblock Essential laminations in {$3$}-manifolds.
\newblock {\em Ann. of Math. (2)}, 130(1):41--73, 1989.

\bibitem[HKM07]{HondaKazezMatic_RightVeeringI}
Ko~Honda, William~H. Kazez, and Gordana Mati\'c.
\newblock Right-veering diffeomorphisms of compact surfaces with boundary.
\newblock {\em Invent. Math.}, 169(2):427--449, 2007.

\bibitem[HKM08]{HondaKazezMatic_RightVeeringII}
Ko~Honda, William~H. Kazez, and Gordana Mati\'c.
\newblock Right-veering diffeomorphisms of compact surfaces with boundary.
  {II}.
\newblock {\em Geom. Topol.}, 12(4):2057--2094, 2008.

\bibitem[HM18]{HeddenMark_FloerFDTC}
Matthew Hedden and Thomas~E. Mark.
\newblock Floer homology and fractional {D}ehn twists.
\newblock {\em Adv. Math.}, 324:1--39, 2018.
\newblock ArXiv:1501.01284 [math.GT].

\bibitem[HS07]{Ham_Song_PseudoAnosov5Braid}
Ji-Young Ham and Won~Taek Song.
\newblock The minimum dilatation of pseudo-{A}nosov 5-braids.
\newblock {\em Experiment. Math.}, 16(2):167--179, 2007.

\bibitem[IK17a]{ItoKawamuro_OpenBookFoliations}
Tetsuya Ito and Keiko Kawamuro.
\newblock Essential open book foliations and fractional {D}ehn twist
  coefficient.
\newblock {\em Geom. Dedicata}, 187:17--67, 2017.

\bibitem[IK17b]{ItoKawamuro_OnAQuestion}
Tetsuya Ito and Keiko Kawamuro.
\newblock On a question of {E}tnyre and {V}an {H}orn-{M}orris.
\newblock {\em Algebr. Geom. Topol.}, 17(1):561--566, 2017.

\bibitem[Jon87]{Jones_MFW}
V.~F.~R. Jones.
\newblock Hecke algebra representations of braid groups and link polynomials.
\newblock {\em Ann. of Math. (2)}, 126(2):335--388, 1987.

\bibitem[Kaw06]{Kawamuro_braidindex}
Keiko Kawamuro.
\newblock The algebraic crossing number and the braid index of knots and links.
\newblock {\em Algebr. Geom. Topol.}, 6:2313--2350, 2006.

\bibitem[Kaw09]{Kawamuro_KR_MFW}
Keiko Kawamuro.
\newblock {K}hovanov-{R}ozansky homology and the braid index of a knot.
\newblock {\em Proc. Amer. Math. Soc.}, 137(7):2459--2469, 2009.

\bibitem[KR13]{KazezRoberts_FDTC}
William~H. Kazez and Rachel Roberts.
\newblock Fractional {D}ehn twists in knot theory and contact topology.
\newblock {\em Algebr. Geom. Topol.}, 13(6):3603--3637, 2013.

\bibitem[Liv04]{Livingston_Comp}
Charles Livingston.
\newblock Computations of the {O}zsv{\'a}th-{S}zab{\'o} knot concordance
  invariant.
\newblock {\em Geom. Topol.}, 8:735--742 (electronic), 2004.
\newblock ArXiv:0311036v3 [math.GT].

\bibitem[Liv17]{Livingston_NotesOnUpsilon}
Charles Livingston.
\newblock Notes on the knot concordance invariant upsilon.
\newblock {\em Algebr. Geom. Topol.}, 17(1):111--130, 2017.

\bibitem[LL19]{LewarkLobb_17_Upsilonlike}
Lukas Lewark and Andrew Lobb.
\newblock Upsilon-like concordance invariants from $sl(n)$ knot cohomology.
\newblock {\em Geom. Topol.}, 23(2):745--780, 2019.
\newblock ArXiv:1707.00891 [math.GT].

\bibitem[LM14]{LaFountainMenasco_14}
Douglas~J. LaFountain and William~W. Menasco.
\newblock Embedded annuli and {J}ones' conjecture.
\newblock {\em Algebr. Geom. Topol.}, 14(6):3589--3601, 2014.

\bibitem[Mal04]{Malyutin_Twistnumber}
A.~V. Malyutin.
\newblock Writhe of (closed) braids.
\newblock {\em Algebra i Analiz}, 16(5):59--91, 2004.

\bibitem[MN03]{MalyutinNetsvetaev_03}
A.~V. Malyutin and N.~Yu. Netsvetaev.
\newblock Dehornoy order in the braid group and transformations of closed
  braids.
\newblock {\em Algebra i Analiz}, 15(3):170--187, 2003.

\bibitem[Mor86]{Morton_SeifertCircles}
H.~R. Morton.
\newblock Seifert circles and knot polynomials.
\newblock {\em Math. Proc. Cambridge Philos. Soc.}, 99(1):107--109, 1986.

\bibitem[Mor88]{Morton_PolynomialsFromBraids}
H.~R. Morton.
\newblock Polynomials from braids.
\newblock In {\em Braids ({S}anta {C}ruz, {CA}, 1986)}, volume~78 of {\em
  Contemp. Math.}, pages 575--585. Amer. Math. Soc., Providence, RI, 1988.

\bibitem[Mur91]{Murasugi_BraidIndexAlternatingLinks}
Kunio Murasugi.
\newblock On the braid index of alternating links.
\newblock {\em Trans. Amer. Math. Soc.}, 326(1):237--260, 1991.

\bibitem[OS03]{OzsvathSzabo_03_KFHandthefourballgenus}
Peter Ozsv{\'a}th and Zolt{\'a}n Szab{\'o}.
\newblock Knot {F}loer homology and the four-ball genus.
\newblock {\em Geom. Topol.}, 7:615--639, 2003.

\bibitem[OSS17]{OSS_2014}
Peter Ozsv{\'a}th, Andr{\'a}s~I. Stipsicz, and Zolt{\'a}n Szab{\'o}.
\newblock Concordance homomorphisms from knot { F}loer homology.
\newblock {\em Advances in Mathematics}, pages 366--426, 2017.

\bibitem[Pla15]{Plamenevskaya_RightVeering}
Olga Plamenevskaya.
\newblock Transverse invariants and right-veering.
\newblock {\em ArXiv e-prints}, 2015.
\newblock ArXiv:1509.01732 [math.GT].

\bibitem[Rud93]{rudolph_QPasobstruction}
Lee Rudolph.
\newblock Quasipositivity as an obstruction to sliceness.
\newblock {\em Bull. Amer. Math. Soc. (N.S.)}, 29(1):51--59, 1993.

\bibitem[RW00]{Rourke_OrderAutomaticMappingClassGroups}
Colin Rourke and Bert Wiest.
\newblock Order automatic mapping class groups.
\newblock {\em Pacific J. Math.}, 194(1):209--227, 2000.

\bibitem[SW00]{Short_OrderingsmappingclassgroupsafterThurston}
Hamish Short and Bert Wiest.
\newblock Orderings of mapping class groups after {T}hurston.
\newblock {\em Enseign. Math. (2)}, 46(3-4):279--312, 2000.

\end{thebibliography}
\end{document}